\documentclass[11pt]{amsart}
\usepackage{amsmath}
\usepackage{amsfonts,amssymb,amsthm, txfonts, pxfonts,amscd}
\usepackage[stretch=10]{microtype}
\usepackage{hyperref}
\def\struckint{\mathop{%
\def\mathpalette##1##2{\mathchoice{##1\displaystyle##2}%
  {##1\textstyle##2}{##1\scriptstyle##2}{##1\scriptscriptstyle##2}}%
\mathpalette
{\vbox\bgroup\baselineskip0pt\lineskiplimit-1000pt\lineskip-1000pt
\halign\bgroup\hfill$}
{##$\hfill\cr{\intop}\cr\diagup\cr\egroup\egroup}%
}\limits}

\newtheorem{proposition}{Proposition}[section]
\newtheorem{theorem}[proposition]{Theorem}
\newtheorem{lemma}[proposition]{Lemma}
\newtheorem{corollary}[proposition]{Corollary}
\newtheorem{assumption}[proposition]{Assumption}

\theoremstyle{definition}
\newtheorem{definition}[proposition]{Definition}

\newtheorem{note}[proposition]{Note}

\newenvironment{pf*}[1]{\medskip \noindent {\em #1.} }{\endproof \medskip}

\newcommand{\xclass}[1]{\langle #1 \rangle}

\newcommand{\xnorm}[1]{ \Vert #1 \Vert }

\newcommand{\zz}[1]{\mathbb #1}

\newcommand{\bSigma}{\overline{\Sigma}}
\newcommand{\Z}{\mathbb{Z}}
\newcommand{\HH}{\mathcal{H}}
\newcommand{\RL}{\mathcal{L}}
\DeclareMathOperator{\Press}{Pr}

\title[Random Walks on Co-Compact Fuchsian Groups]{Random Walks on Co-Compact Fuchsian Groups}
\author{S\'ebastien Gou\"ezel and Steven P. Lalley}
\address{IRMAR, CNRS UMR 6625,
Universit\'e de Rennes 1, 35042 Rennes, France}
\email{sebastien.gouezel@univ-rennes1.fr}
\address{University of Chicago\\ Department
of Statistics \\ 5734
University Avenue \\
Chicago IL 60637}
\email{lalley@galton.uchicago.edu}
\date{\today}
\thanks{Second author supported by NSF grant DMS  - 0805755}

\begin{document}

\begin{abstract}
It is proved that the Green's function of a symmetric finite range
random walk on a co-compact Fuchsian group decays exponentially in
distance at the radius of convergence $R$. It is also shown that
Ancona's inequalities extend to $R$, and therefore that the Martin
boundary for $R-$potentials coincides with the natural geometric
boundary $S^{1}$, and that the Martin kernel is uniformly H\"older
continuous. Finally, this implies a local limit theorem for the
transition probabilities: in the aperiodic case, $p^n(x,y)\sim
C_{x,y}R^{-n}n^{-3/2}$.
\end{abstract}

\maketitle

\section{Introduction}\label{sec:introduction}

\subsection{Green's function and Martin boundary}\label{ssec:grfmb}
A (right) \emph{random walk} on a countable group $\Gamma$ is a discrete-time
Markov chain $\{X_{n} \}_{n\geq 0}$ of the
form
\begin{equation*}
	X_{n}=x\xi_{1}\xi_{2}\dotsb \xi_{n}
\end{equation*}
where $\xi_{1},\xi_{2},\dotsc $ are independent, identically
distributed $\Gamma -$valued random variables. The distribution of
$\xi_{i}$ is the \emph{step distribution} of the random walk. The
random walk is said to be \emph{symmetric} if its step distribution
is invariant under the mapping $x\mapsto x^{-1}$, and
\emph{finite-range} if the step distribution has finite support. The
\emph{Green's function} is  the generating function of the transition
probabilities: for $x,y\in \Gamma$ and $0\leq r < 1$ it is defined by
the absolutely convergent series
\begin{equation}\label{eq:green}
	G_{r} (x,y):=\sum_{n=0}^{\infty} P^{x}\{X_{n}=y \}r^{n} =G_{r} (1,x^{-1}y);
\end{equation}
here $P^{x}$ is the probability measure on path space governing the
random walk with initial point $x$.  If the random walk is
irreducible (that is, if the semigroup generated by the support of
the step distribution is $\Gamma$) then the radius of convergence $R$
of the series \eqref{eq:green} is the same for all pairs $x,y$.
Moreover, if the random walk is symmetric, then $1/R$ is the
\emph{spectral radius} of the transition operator. By a fundamental
theorem of Kesten \cite{kesten}, if the group $\Gamma$ is finitely
generated and nonamenable then $R>1$. Moreover, in this case
the Green's function is finite at its radius of convergence (cf.\ %
\cite{woess:book}, ch.~2): for all $x,y\in\Gamma$,
\begin{equation}\label{eq:finiteAtR}
	G_{R} (x,y)<\infty.
\end{equation}

The Green's function is of central importance in the study of random
walks. Clearly, it encapsulates information about the transition
probabilities; in Theorem~\ref{theorem:asymptotics}, we show that the
local asymptotic behavior of the transition probabilities can be
deduced from the singular behavior of the Green's function at its
radius of convergence. The Green's function is also the key to the
potential theory associated with the random walk: in particular, it
determines the Martin boundary for $r-$potential theory. A prominent
theme in the study of random walks on nonabelian groups has been the
relationship between the geometry of the group and the nature of the
Martin boundary. A landmark result here is a theorem of
Ancona~\cite{ancona} describing the Martin boundary for random walks
with finitely supported step distributions on \emph{hyperbolic}
groups: Ancona proves that for every $r\in (0,R)$ the Martin boundary
for $r-$potential theory coincides with the \emph{geometric} (Gromov)
boundary, in a sense made precise below.  (Series~\cite{series-rw}
had earlier established this in the special case $r=1$ when the group
is co-compact Fuchsian. See also \cite{anderson-schoen} and
\cite{ancona:annals} for related results concerning Laplace-Beltrami
operators on Cartan manifolds.)

It is natural to ask whether Ancona's theorem extends to $r=R$, that
is, if the Martin boundary is stable (see \cite{picardello-woess} for
the terminology) through the entire range $(0,R]$. One of the main
results of this paper (Theorem~\ref{theorem:1}) provides an
affirmative answer in the special case of symmetric, finite-range
random walk on a co-compact Fuchsian group, i.e., a co-compact,
discrete subgroup of $PSL (2,\zz{R})$. Any co-compact Fuchsian group
acts as a discrete group of isometries of the hyperbolic disk, and so
its Cayley graph can be embedded quasi-isometrically in the
hyperbolic disk; this implies that its Gromov boundary is the circle
$S^{1}$ at infinity.

\begin{theorem}\label{theorem:martinBoundary}
For any symmetric, irreducible, finite-range random walk on a
co-compact Fuchsian group $\Gamma $, the Martin boundary for
$R-$potentials coincides with the geometric boundary $S^{1}=\partial
\Gamma$. Moreover, all elements of the Martin boundary are minimal.
\end{theorem}

This assertion means that (a) for every geodesic ray
$y_{0},y_{1},y_{2},\dotsc$ in the Cayley graph that converges to a
point $\zeta \in \partial \Gamma$ and for every $x\in \Gamma$,
\begin{equation}\label{eq:martinConvergence}
	\lim_{n \rightarrow \infty} \frac{G_{R} (x,y_{n})}{G_{R}
	(1,y_{n})}=K_{R}(x,\zeta )=K (x,\zeta)
\end{equation}
exists; (b) for each $\zeta \in \partial \Gamma$ the function
$K_{\zeta} (x):= K(x,\zeta )$ is a minimal, positive $R-$harmonic
function of $x$; (c) for distinct points $\zeta ,\zeta '\in \partial
\Gamma$ the functions $K_{\zeta}$ and $K_{\zeta '}$ are different; and
(d) the topology of pointwise convergence on $\{K_{\zeta} \}_{\zeta
\in \partial \Gamma}$ coincides with the usual topology on $\partial
\Gamma=S^{1}$.

Our  results also yield explicit rates for the convergence
\eqref{eq:martinConvergence}, and imply that the Martin kernel $K_{r}
(x,\zeta)$ is \emph{H\"older} continuous in $\zeta$ relative to the
usual Euclidean metric (or any visual metric --- see
\cite{benakli-kapovich} for the definition) on $S^{1}=\partial
\Gamma$.

\begin{theorem}\label{theorem:holderMartinKernel}
For any symmetric, irreducible, finite-range random walk on a
co-compact Fuchsian group $\Gamma $, there exists $\varrho <1$ such
that for every $1\leq r\leq R$ and every geodesic ray $1=y_{0},y_{1},
y_{2},\dotsc$ converging to a point $\zeta \in \partial \Gamma$,
\begin{equation}\label{eq:convergenceRate}
	\Bigg| \frac{G_{r} (x,y_{n})}{G_{r} (1,y_{n})}-K_{r}
	(x,\zeta)\Bigg|
	\leq C_{x}\varrho^{n}.
\end{equation}
The constants $C_{x}<\infty$ depend on $x\in \Gamma$ but not on
$r\leq R$. Consequently, for each $x\in \Gamma$ and $r\leq R$ the
function $\zeta \mapsto K_{r} (x,\zeta)$ is H\"older continuous in
$\zeta$ relative to the Euclidean metric on $S^{1}=\partial \Gamma$,
for some exponent not depending on $r\leq R$.
\end{theorem}

The exponential convergence \eqref{eq:convergenceRate} and the H\"older
continuity of the Martin kernel for $r=1$ were established by Series
\cite{series} for random walks on Fuchsian groups. Similar results
for the Laplace-Beltrami operator on negatively curved Cartan
manifolds were proved by Anderson and Schoen \cite{anderson-schoen}.
The methods of \cite{anderson-schoen} were adapted by Ledrappier
\cite{ledrappier:review} to prove that Series' results extend to all
random walks on a free group, and Ledrappier's proof was extended by
Izumi, Neshvaev, and Okayasu \cite{izumi} to prove that for a random
walk on a non-elementary hyperbolic group the Martin kernel $K_{1}
(x,\zeta)$ is H\"older continuous in $\zeta$. All of these proofs rest
on inequalities of the type discussed in
section~\ref{ssec:anconaIneq} below.  Theorem~\ref{theorem:1} below
asserts (among other things) that similar estimates are valid for all
$G_{r}$ \emph{uniformly for} $r\leq R$. Given these, the proof of
\cite{izumi} applies almost verbatim to establish
Theorem~\ref{theorem:holderMartinKernel}. We will give some
additional details in Paragraph~\ref{subsec:fellow_traveling}.

\subsection{Ancona's boundary Harnack inequalities}\label{ssec:anconaIneq}

The crux of Ancona's argument in \cite{ancona} was a system of
inequalities that assert, roughly, that the Green's function $G_{r}
(x,y)$ is nearly submultiplicative in the arguments $x,y\in \Gamma$.
Ancona \cite{ancona} proved that such inequalities always hold for
$r<R$: in particular, he proved, for a random walk with finitely
supported step distribution on a hyperbolic group, that for each
$r<R$ there is a constant $C_{r}<\infty$ such that for every geodesic
segment $x_{0}x_{1}\dotsb x_{m}$ in (the Cayley graph of) $\Gamma$,
\begin{equation}\label{eq:ancona}
	G_{r} (x_{0},x_{m})\leq C_{r} G_{r} (x_{0},x_{k})G_{r}
	(x_{k},x_{m}) \qquad \forall \, 1\leq k\leq m.
\end{equation}
His argument depends in an essential way on the hypothesis $r<R$
(cf.\ his Condition (*)), and it leaves open the possibility that the
constants $C_{r}$ in the inequality \eqref{eq:ancona} might blow up
as $r \rightarrow R$. For finite-range random walk on a \emph{free}
group it can be shown, by direct calculation, that the constants
$C_{r}$ remain bounded as $r \rightarrow R$, and that the
inequalities \eqref{eq:ancona} remain valid at $r=R$ (cf.\
\cite{lalley:frrw}). The following result asserts that the same is
true for symmetric random walks on a co-compact Fuchsian group.

\begin{theorem}\label{theorem:1}
For any symmetric, irreducible, finite-range random walk on a
co-compact Fuchsian group $\Gamma $,
\begin{enumerate}
\item [(A)] the Green's function $G_{R} (1,x) $ decays exponentially
in $|x|:=d (1,x)$; and
\item [(B)] Ancona's inequalities \eqref{eq:ancona} hold for all
$r\leq R$, with a constant $C$ independent of $r$.
\end{enumerate}
\end{theorem}

\begin{note}\label{note:1}
Here  and throughout the paper $d (x,y)$ denotes the distance between
the vertices $x$ and $y$ in the Cayley graph $G^{\Gamma}$,
equivalently, distance in the word metric. \emph{Exponential decay}
of the Green's function means \emph{uniform} exponential decay in all
directions, that is, there are constants $C<\infty$ and $\varrho <1$
such that for all $x,y\in \Gamma$,
\begin{equation}
\label{eq:expDecay}
	G_{R} (x,y)\leq C\varrho^{d (x,y)}.
\end{equation}
A very simple argument (see Lemma \ref{lemma:backscattering} below)
shows that for a symmetric random walk on any nonamenable group
$G_{R} (1,x) \rightarrow 0$ as $|x| \rightarrow \infty$. Given this,
it is routine to show that exponential decay of the Green's function
follows from Ancona's inequalities. However, we will argue in the
other direction, first providing an independent proof of exponential
decay in subsection~\ref{ssec:decay}, and then deducing Ancona's
inequalities from it in section~\ref{sec:ancona}.
\end{note}

\begin{note}\label{note:hamenstaedt}
Theorem~\ref{theorem:1} (A) is a discrete analogue of one of the main
results (Theorem B) of Hamenstaedt \cite{hamenstaedt} concerning the
Green's function of the Laplacian on the universal cover of a compact
negatively curved manifold.  Unfortunately, Hamenstaedt's proof
appears to have a serious error.\footnote{The error is in the proof of
Lemma~3.1: The claim is made that a lower bound on a finite measure
implies a lower bound for its Hausdorff-Billingsley dimension relative
to another measure. This is false -- in fact such a lower bound on
measure implies an \emph{upper} bound on its Hausdorff-Billingsley
dimension. } The approach taken here bears no resemblance to that of
\cite{hamenstaedt}.
\end{note}

Theorem ~\ref{theorem:1} is proved in sections~\ref{sec:apriori} and
\ref{sec:ancona} below. The argument uses  the \emph{planarity} of
the Cayley graph in an essential way.
It also relies on the  simple estimate
\begin{equation*}
	\lim_{ |x| \rightarrow \infty}G_{R} (1,x)=0,
\end{equation*}
that we derive from the symmetry of the random walk.  While this
estimate is not true in general without the symmetry assumption, we
nevertheless conjecture that Ancona's inequalities and the
identification of the Martin boundary at $r=R$ hold in general.

\subsection{Decay at infinity of the Green's
function}\label{ssec:decayRate}

Neither Ancona's result nor Theorem \ref{theorem:1} gives any
information about how the uniform exponential decay rate $\varrho$
depends on the step distribution of the random walk. In fact, the
Green's function $G_{r} (1,x)$ decays at different rates in different
directions $x \rightarrow \partial \Gamma$. To quantify the overall
decay, consider the behavior of the Green's function over the entire
sphere $S_{m}$ of radius $m$ centered at $1$ in the Cayley graph
$G^{\Gamma}$.  If $\Gamma$ is a nonelementary Fuchsian group then the
cardinality of the sphere $S_{m}$ grows exponentially in $m$ (see
Corollary~\ref{corollary:sphereGrowth} in section~\ref{sec:cannon}),
that is, there exist constants $C>0$ and $\zeta >1$ such that as $m
\rightarrow \infty$,
\begin{equation*}
	 |S_{m}| \sim C \zeta^{m}.
\end{equation*}

\begin{theorem}\label{theorem:2}
For any symmetric, irreducible, finite-range random walk on a
co-compact Fuchsian group $\Gamma $,
\begin{equation}\label{eq:backscatterA}
	 \lim_{m \rightarrow \infty}
	 \sum_{x\in S_{m}}G_{R} (1,x)^{2} = C>0
\end{equation}
exists and is finite, and
\begin{equation}\label{eq:lp}
	\# \{x\in \Gamma \, : \, G_{R} (1,x)\geq \varepsilon \}
	\asymp
	\varepsilon^{-2}
\end{equation}
as $\varepsilon  \rightarrow 0$. (Here $\asymp$ means that the ratio
of the two sides remains bounded away from $0$ and $\infty$.)
\end{theorem}

The proof is carried out in
sections~\ref{sec:thermo}--\ref{sec:pressureEvaluation} below
(cf.\ Propositions~\ref{proposition:absolutelyCont} and
\ref{proposition:pressureEqualsZero}), using the fact that any
hyperbolic group has an \emph{automatic structure}
\cite{ghys-deLaHarpe}. The automatic structure will permit us to use
the theory of \emph{Gibbs states} and \emph{thermodynamic formalism}
of Bowen \cite{bowen}, ch.~1.
Theorem~\ref{theorem:holderMartinKernel} is essential for this, as the
theory developed in \cite{bowen} applies only to H\"older continuous
functions.

It is likely that $\asymp$ can be replaced by $\sim$ in \eqref{eq:lp}.
There is a simple heuristic argument that suggests why the sums
$\sum_{x\in S_{m}}G_{R} (1,x)^{2}$ should remain bounded as $m
\rightarrow \infty$: Since the random walk is $R-$transient, the
contribution to $G_{R} (1,1)<\infty$ from random walk paths that visit
$S_{m}$ and then return to $1$ is bounded (by $G_{R} (1,1)$).  For any
$x\in S_{m}$, the term $G_{R}(1,x)^{2}/G_{R}(1,1)$ is the contribution
to $G_{R} (1,1)$ from paths that visit $x$ before returning to $1$.
Thus, if $G_{R} (1,x)$ is not substantially larger than
\[
	\sum_{n=1}^{\infty} P^{1}\{X_{n}=x \; \text{and} \;\tau (m)=n\}R^{n},
\]
where $\tau (m)$ is the time of the first visit to $S_{m}$, then the
sum in \eqref{eq:backscatterA} should be of the same order of magnitude
as the total contribution to $G_{R} (1,1)<\infty$ from random walk
paths that visit $S_{m}$ and then return to $1$. Of course, the
difficulty in making this heuristic argument rigorous is that \emph{a
priori} one does not know that paths that visit $x$ are likely to be
making their first visits to $S_{m}$; it is Ancona's inequality
\eqref{eq:ancona}  that ultimately fills the gap.

\begin{note}\label{note:greenOnSphere}
A simple argument shows that for $r>1$ the sum of the Green's
function on the sphere $S_{m}$, unlike the sum of its square,
explodes as $m \rightarrow \infty$. Fix $1< r\leq R$ and $m\geq 1$.
Let $C_0$ bound the size of the jumps of the random walk, and let
$\tilde S_m$ be the set of points with $d(1,x)\in [m, m+C_0)$. Since
$X_{n}$ is transient, it will, with probability one, eventually visit
the annulus $\tilde S_{m}$. The minimum number of steps needed to
reach $\tilde S_{m}$ is at least $m/C_0$. Hence,
\begin{align*}
	\sum_{x\in \tilde S_{m}}G_{r}(1,x)&=\sum_{n=m/C_0}^{\infty} \sum_{x\in
	\tilde S_{m}} P^{1}\{X_{n}=x \}r^{n}\\
	&\geq r^{m/C_0}\sum_{n=m/C}^{\infty} P^{1}\{X_{n}\in \tilde S_{m} \}\\
	&\geq r^{m/C_0} P^{1}\{X_{n}\in \tilde S_{m} \text{ for some}\; n\}\\
	&=r^{m/C_0}.
\end{align*}
Hence, $\sum_{x\in \tilde S_{m}}G_{r}(1,x)$ diverges. The divergence
of $\sum_{x\in S_{m}}G_{r}(1,x)$ readily follows if the random walk
is irreducible.
\end{note}

\begin{note}\label{note:ledrappier}
There are some precedents for the result \eqref{eq:backscatterA}.
Ledrappier \cite{ledrappier:renewal} has shown that for Brownian
motion on the universal cover of a compact Riemannian manifold of
negative curvature, the integral of the Green's function $G_{1}
(x,y)=\int_{0}^{\infty}p_{t} (x,y)\,dt$ over the sphere $S(\varrho
,x)$ of radius $\varrho $ centered at a fixed point $x$ converges as
$\varrho \rightarrow \infty$ to a positive constant $C$ independent
of $x$. Hamenstaedt \cite{hamenstaedt} proves in the same context
that the integral of $G_{R}^{2}$ over $S (\varrho ,x)$ remains
bounded as the radius $\varrho \rightarrow \infty$. Our arguments
(see Note~\ref{note:theta} in sec.~\ref{sec:thermo}) show that for
finite range irreducible random walk on a co-compact Fuchsian group
the following is true: for each value of $r$ there exists a power
$1\leq \theta =\theta (r)\leq 2$ such that
\[
	 \lim_{m \rightarrow \infty}
	 \sum_{x\in S_{m}}G_{r} (1,x)^{\theta } = C_{r}>0.
\]
\end{note}

\subsection{Critical exponent for the  Green's
function}\label{ssec:criticalExponent}

Theorem \ref{theorem:2} implies that
the behavior of the Green's function $G_{R}(x,y)$ at the radius of
convergence as $y$ approaches the geometric boundary is intimately
related to the behavior of  $G_{r} (x,y)$ as $r\uparrow R$. The
connection between the two is rooted in the following set of
differential equations.

\begin{proposition}\label{proposition:GPrime}
For any random walk on any discrete group, the Green's functions
satisfy
\begin{equation}\label{eq:GPrime}
	\frac{d}{dr}G_{r} (x,y)=r^{-1}\sum_{z\in \Gamma}
			  G_{r} (x,z)G_{r} (z,y) -r^{-1}G_{r} (x,y)
	\quad \forall \; 0\leq r <R.
\end{equation}
\end{proposition}

Although the proof is elementary (cf.\ section \ref{ssec:GFs} below)
these differential equations have not (to our knowledge) been observed
before.  Theorem~\ref{theorem:2} implies that the sum in
equation~\eqref{eq:GPrime} blows up as $r \rightarrow R-$; this is
what causes the singularity of $r\mapsto G_{r} (1,1)$ at $r=R$.  The
rate at which the sum blows up determines the \emph{critical
exponent} for the Green's function, that is, the exponent $\alpha $
for which $G_{R} (1,1) -G_{r} (1,1)\sim C (R-r)^{\alpha}$. The
following theorem asserts that the critical exponent is $1/2$.

\begin{theorem}\label{theorem:criticalExponent}
For any symmetric, irreducible, finite-range random walk on a
co-compact Fuchsian group $\Gamma $, there exist constants
$C_{x,y}>0$ such that as $r \rightarrow R-$,
\begin{align}\label{eq:criticalExponent}
	G_{R} (x,y)-G_{r} (x,y) &\sim C_{x,y}\sqrt{R-r} \quad \text{and}\\
	dG_{r} (x,y)/dr &\sim \frac{1}{2} C_{x,y}/\sqrt{R-r}.
\end{align}
\end{theorem}

The proof of Theorem~\ref{theorem:criticalExponent} is given in
section~\ref{sec:criticalExp}. Like the proof of
Theorem~\ref{theorem:2}, it uses the existence of an automatic
structure and the attendant thermodynamic formalism. It also relies
critically on the conclusion of Theorem~\ref{theorem:2}, which
determines the value of the key thermodynamic variable.

The behavior of the generating function $G_{r} (1,1)$ in the
neighborhood of the singularity $r=R$ is of interest because it
reflects the asymptotic behavior
of the coefficients $P^{1}\{X_{n}=1 \}$ as $n \rightarrow \infty$.
In section~\ref{sec:asymptotics} we will show that
Theorem~\ref{theorem:criticalExponent}, in conjunction with Karamata's
Tauberian Theorem, implies the following \emph{local limit theorem}.

\begin{theorem} \label{theorem:localLimit}
For any symmetric, irreducible, finite-range, aperiodic random walk
on a co-compact Fuchsian group with spectral radius $R^{-1}$, there
exist constants $C_{x,y}>0$ such that for all $x,y\in \Gamma$,
  \begin{equation}
  \label{eq:llt}
  p^n(x,y)\sim C_{x,y} R^{-n}n^{-3/2}.
  \end{equation}
If the random walk  is not aperiodic, these asymptotics hold for even
(resp.~odd) $n$ if the distance from $x$ to $y$ is even (resp.~odd).
\end{theorem}

According to a theorem of Bougerol \cite{bougerol}, the transition
probability densities of a random walk on a semi-simple Lie group
follow a similar asymptotic law, provided the step distribution is
rapidly decaying and has an absolutely continuous component with
respect to the Haar measure. The exponential decay rate depends on
the step distribution, but the critical exponent (the power of $n$ in
the asymptotic formula, in our case $3/2$) depends only on the rank
and the number of positive, indivisible roots of the group.
Theorem~\ref{theorem:localLimit} shows that -- at least for $SL
(2,\zz{R})$ -- the critical exponent is inherited by a large class of
co-compact discrete subgroups. For random walks on free groups
\cite{gerl-woess}, \cite{lalley:frrw}, most free products
\cite{woess:freeProducts}, and virtually free groups that are not
virtually cyclic (including $SL_{2}(\zz{Z})$), \cite{woess:local},
\cite{lalley:rwrl}, \cite{nagnibeda-woess}, \cite{woess:preprint}
local limit theorems of the form \eqref{eq:llt} have been known for
some time. In all of these cases the Green's functions are algebraic
functions of $r$.  We expect (but cannot prove) that for symmetric,
finite-range random walks on co-compact Fuchsian groups the Green's
functions are \emph{not} algebraic. However, we will prove in
section~\ref{sec:asymptotics} (Theorem~\ref{theorem:spectrum}) that
the Green's function admits an analytic continuation to a doubly slit
plane $\zz{C}\setminus ([R,\infty)\cup (-\infty ,-R
(1+\varepsilon)])$, and Theorem~\ref{theorem:criticalExponent}
implies that if the singularity at $r=R$ is a branch point then it
must be of order $2$. If it could be shown that the singularity is
indeed a branch point then it would follow that the transition
probabilities have complete asymptotic expansions in powers of
$n^{-1/2}$.

The most important step in our program -- the proof of Ancona's
inequalities (Theorem~\ref{theorem:1}) -- depends heavily on the
planarity of the Cayley graph of the group. The derivation of the
subsequent results (including
Theorems~\ref{theorem:2},~\ref{theorem:criticalExponent}
and~\ref{theorem:localLimit}) uses thermodynamic formalism, which is
possible since the Markov automaton associated to Fuchsian groups is
recurrent. Nevertheless, most of our techniques apply in arbitrary
hyperbolic groups. We expect that our results should hold (maybe in
weaker forms) in this broader context, but the proofs would require
significant new ideas. The only results we are aware of in this
direction are the following polynomial estimates (that were suggested
to us by an anonymous referee): for any symmetric, irreducible,
finite-range, aperiodic random walk in an hyperbolic group with
spectral radius $R^{-1}$, one has
  \begin{equation*} CR^{-n} n^{-3}
  \leq p^n(1,1) \leq C R^{-n}n^{-1}.
  \end{equation*}
The upper bound follows from the fact that the sequence
$R^{2n}p^{2n}(1,1)$ is non-increasing and summable, hence $o(1/n)$.
For the lower bound one uses Property RD (see for instance
\cite{chatterji}, and especially Proposition 1.7 there to get
quantitative estimates for hyperbolic groups): there exists a
constant $C>0$ such that, for any function $f$ supported in the ball
$B(1,n)$, the operator norm of $g\mapsto f\star g$ on
$\ell^2(\Gamma)$ is bounded by $C n^{3/2}\lVert f \rVert_{\ell^2}$.
We apply this estimate to $f(x)=p^n(1,x)$: the norm of the
convolution with $f$ is $R^{-n}$, while $\lVert f\rVert_{\ell^2} =
(p^{2n}(1,1))^{1/2}$. Therefore, Property RD gives $R^{-n} \leq C
n^{3/2} (p^{2n}(1,1))^{1/2}$, proving the desired lower bound.

For random walks with non-symmetric step distributions, the local
limit theorem holds in free groups.  We expect it to hold also in
general hyperbolic groups. However, we do not even know how to prove
polynomial bounds for this case.

\section{Green's function: preliminaries}\label{sec:preliminaries}

Throughout this section, $X_{n}$ is a symmetric, finite-range
irreducible random walk on a finitely generated, nonamenable group
$\Gamma$ with (symmetric) generating set $A$. Let $S$ denote the
support of the step distribution of the random walk. We assume
throughout that $S$ is finite, and hence contained
in a ball $B(1,C_0)$ for some $C_0\geq 1$.

\subsection{Green's function as a sum over paths}\label{ssec:GFs}

The Green's function $G_{r} (x,y)$ defined by \eqref{eq:green} has an
obvious interpretation as a sum over paths from $x$ to $y$. (Note:
Here and in the sequel a \emph{path} in $\Gamma$ is just a sequence
$x_{n}$ of vertices in the Cayley graph $G^{\Gamma}$ with
$x_{n}^{-1}x_{n+1}\in S$ for all $n$). Denote by $\mathcal{P} (x,y)$
the set of all paths $\gamma$ from $x$ to $y$, and for any such path
$\gamma = (x_{0},x_{1},\dotsc ,x_{m})$ define the \emph{weight}
\begin{equation*}
	w_{r} (\gamma):=r^{m}\prod_{i=0}^{m-1}p (x_{i},x_{i+1}).
\end{equation*}
Then
\begin{equation}\label{eq:greenByPath}
	G_{r} (x,y)=\sum_{\gamma \in \mathcal{P} (x,y)}w_{r} (\gamma).
\end{equation}
Since the step distribution $p (x)=p (x^{-1})$ is symmetric with
respect to inversion, so is the weight function $\gamma \mapsto w_{r}
(\gamma)$: if $\gamma^{R}$ is the reversal of the path $\gamma$, then
$w_{r} (\gamma^{R})=w_{r} (\gamma)$. Consequently, the Green's
function is symmetric in its arguments:
\begin{equation}\label{eq:GreenSymmetry}
	G_{r} (x,y)=G_{r} (y,x).
\end{equation}
Also, the weight function is multiplicative with respect to
concatenation of paths, that is, $w_{r} (\gamma \gamma ')=w_{r}
(\gamma)w_{r} (\gamma ')$. Since the random walk is irreducible,
every generator of the group can be reached by the random walk in
finite time with positive probability. It follows that the Green's
function satisfies a system of \emph{Harnack inequalities}: There
exists a constant $C<\infty$ such that for each $0<r\leq R$ and all
group elements $x,y,z$,
\begin{equation}\label{eq:harnack}
	G_{r} (x,z)\leq C^{d (y,z)}G_{r} (x,y).
\end{equation}

\begin{proof}[Proof of Proposition~\ref{proposition:GPrime}]
This is a routine calculation based on the representation
\eqref{eq:greenByPath} of the Green's function as a sum over paths.
Since all terms in the power series representation of the Green's
function have nonnegative coefficients, interchange of $d/dr$ and
$\sum_{\gamma}$ is permissible, so
\[
	\frac{d}{dr}G_{r} (x,y)=\sum_{\gamma \in \mathcal{P} (x,y)}
	\frac{d}{dr} w_{r} (\gamma).
\]
  If $\gamma$ is a path from $1$ to $x$
of length $m$, then the derivative with respect to $r$ of the weight
$w_{r} (\gamma)$ is $mw_{r} (\gamma)/r$, so $dw_{r} (\gamma)/dr$
contributes one term of size $w_{r} (\gamma)/r$ for each vertex
visited by $\gamma$ after its first step.  This, together with the
multiplicativity of $w_{r}$, yields the identity \eqref{eq:GPrime}.
\end{proof}

\subsection{First-passage generating functions}\label{ssesc:fpGF}
Other useful generating functions can be obtained by summing path
weights over different sets of paths. Two classes of such generating
functions that will be used below are the \emph{restricted Green's
functions} and the \emph{first-passage} generating functions (called
the \emph{balayage} by Ancona \cite{ancona}) defined as follows. Fix
a  set of vertices $\Omega \subset \Gamma $, and for any  $x,y\in
\Gamma$ let $\mathcal{P}(x,y;\Omega)$ be the set of all paths from
$x$ to $y$ that remain in the region $\Omega$ at all except the
initial and final points. Define
\begin{align*}
	G_{r} (x,y;\Omega)&=\sum_{\mathcal{P} (x,y;\Omega )}w_{r}
	(\gamma),\quad \text{and}\\
    F_{r} (x,y)&=G_{r} (x,y;\Gamma \setminus \{y \}).
\end{align*}
Thus, $F_{r} (x,y)$, the \emph{first-passage generating function}, is
the sum over all paths from $x$ to $y$ that first visit $y$ on the
last step. This generating function has the alternative representation
\begin{equation*}
	F_{r} (x,y)=E^{x}r^{\tau (y)}
\end{equation*}
where $\tau (y)$ is the time of the first visit to $y$ by the random
walk $X_{n}$, and the expectation extends only over those sample paths
such that $\tau (y)<\infty$.
%
%
Finally, since any visit to $y$ by a path started at $x$ must follow a
\emph{first} visit to $y$,
\begin{equation}\label{eq:GxFxG}
	G_{r} (x,y)=F_{r} (x,y)G_{r} (1,1).
\end{equation}
Therefore, since $G_{r}$ is symmetric in its arguments, so is $F_{r}$.

\begin{lemma}\label{lemma:backscattering}
\[
	\lim_{n \rightarrow \infty}\max_{\{x\in \Gamma \,:\, |x|=n\}}
	G_{R} (1,x)=0.
\]
\end{lemma}

\begin{proof}
If $\gamma$ is a path from $1$ to $x$, and $\gamma '$ a path from $x$
to $1$, then the concatenation $\gamma \gamma '$ is a path from $1$
back to $1$. Furthermore, since any path from $1$ to $x$ or back must
make at least $|x|/C_0$ steps, the length of $\gamma \gamma '$ is at
least $2|x|/C_0$. Consequently, by symmetry,
\begin{equation}\label{eq:back-forth}
	F_{R} (1,x)^{2}G_{R} (1,1)\leq \sum_{n=2|x|/C_0}^{\infty}
	      P^{1}\{X_{n}=1 \}R^{n}.
\end{equation}
Since $G_{R} (1,1)<\infty$, by nonamenability of the group $\Gamma$,
the tail-sum on the right side of inequality \eqref{eq:back-forth}
converges to $0$ as $|x| \rightarrow \infty$, and so $F_{R}
(1,x)\rightarrow 0$ as $|x| \rightarrow \infty$. Consequently, by
\eqref{eq:GxFxG}, so does $G_{R} (1,x)$.
\end{proof}
Several variations on this argument will be used later.

\subsection{Subadditivity and the random walk
metric}
\label{ssec:superadditivity}
A path from $x$ to $z$ that visits $y$ can be uniquely split into a path
from $x$ to $y$ with first visit to $y$ at the last point, and a path
from $y$ to $z$. Consequently, by the Markov property (or
alternatively the path representation \eqref{eq:greenByPath} and the
multiplicativity of the weight function $w_{r}$), $G_r(x,z)\geq
F_r(x,y)G_r(y,z)$. Since $G_r(x,z)=F_r(x,z)G_r(1,1)$ and
$G_r(y,z)=F_r(y,z)G_r(1,1)$, we deduce that the function $-\log F_{r}
(x,y)$ is \emph{subadditive}:

\begin{lemma}\label{lemma:superadditivity}
For each $r\leq R$ the first-passage generating functions  $F_{r}
(x,y)$ are super-- multiplicative, that is, for any group elements
$x,y,z$,
\begin{equation*}
	F_{r} (x,z)\geq F_{r} (x,y)F_{r} (y,z) .
\end{equation*}
\end{lemma}

Together with Kingman's subadditive ergodic theorem, this implies that
the Green's function $G_{r} (1,x)$ must decay  at a fixed
exponential rate along suitably chosen trajectories. For instance, if
\begin{equation*}
	 Y_{n}=\xi_{1}\xi_{2}\dotsb \xi_{n}
\end{equation*}
where $\xi_{n}$ is an ergodic Markov chain on the alphabet $A$, or on
the set $A^{K}$ of words of length $K$, then Kingman's theorem implies
that
\begin{equation}\label{eq:kingmanRW}
	\lim n^{-1}\log G_{r} (1,Y_{n}) =\alpha  \quad \text{a.s.}
\end{equation}
where $\alpha$ is a constant depending only on $r$ and the transition
probabilities of the underlying Markov chain. More generally, if
$\xi_{n}$ is a suitable ergodic stationary process, then
\eqref{eq:kingmanRW} will hold. 
Super-multiplicativity of the
Green's function also implies the following.

\begin{corollary}\label{corollary:RWMetric}
The function $d_{G} (x,y):=\log F_{R} (x,y)$ is a metric on $\Gamma$.
\end{corollary}

\begin{proof}
The triangle inequality is immediate from
Lemma~\ref{lemma:superadditivity}, and  symmetry
$d_{G}(x,y)=d_{G} (y,x)$  follows from the corresponding symmetry
property \eqref{eq:GreenSymmetry} of the Green's function. Thus, to
show that $d_{G}$ is a metric (and not merely a pseudo-metric) it
suffices to show that if $x\not =y$ then $F_{R} (x,y)<1$.  But this
follows from the fact \eqref{eq:finiteAtR} that the Green's function
is finite at the spectral radius, because the path representation
implies that
\[
	G_{R} (x,x)\geq 1+F_{R} (x,y)^{2} +F_{R} (x,y)^{4} +\dotsb .
    \qedhere
\]
\end{proof}

Call  $d_{G}$  the \emph{Green metric} associated with $F_R$. The
Green metric associated with $F_1$ has been used by a number of
authors,  for instance in~\cite{blachere_brofferio},~\cite{BHM:1}
and~\cite{BHM:2}. The Harnack inequalities imply that $d_{G}$ is
dominated by a constant multiple of the word metric $d$. In general,
there is no domination in the other direction, unless one puts
additional restrictions on the Green function:

\begin{proposition}\label{proposition:qi}
If the Green's function decays exponentially in $d (x,y)$ (that is,
if inequality \eqref{eq:expDecay} holds for all $x,y\in \Gamma$),
then the Green metric $d_{G}$ and the word metric $d$ on ${\Gamma}$
are equivalent, that is, there are constants $0<C_{1}<C_{2}<\infty$
such that for all $x,y\in \Gamma $,
\begin{equation}\label{eq:qi}
	C_{1}d (x,y)\leq d_{G} (x,y)\leq C_{2}d (x,y).
\end{equation}
\end{proposition}

Even when the group $\Gamma$ is hyperbolic, this proposition does not
imply in general that $d_G$ is hyperbolic, since the quasi-isometry
invariance of hyperbolicity only holds for geodesic spaces, while
$(\Gamma, d_G)$ is not geodesic in general. Nevertheless, Theorem 1.1
in \cite{BHM:2} shows that $d_G$ is hyperbolic when Ancona's
inequalities~\eqref{eq:ancona} are satisfied for $r=R$.

\subsection{Green's function and branching random
walks}\label{ssec:brw} There is a simple interpretation of the Green's
function $G_{r} (x,y)$ in terms of the occupation statistics of
\emph{branching random walks}. A branching random walk is built using
a probability distribution $\mathcal{Q}=\{q_{k} \}_{k\geq 0}$ on the
nonnegative integers, called the \emph{offspring distribution},
together with the step distribution $\mathcal{P}:=\{ p(x,y)=p
(x^{-1}y)\}_{x,y\in \Gamma }$ of the
underlying random walk, according to the following rules: At each time
$n\geq 0$, each particle fissions and then dies, creating a random
number of offspring with distribution $\mathcal{Q}$; the offspring
counts for different particles are mutually independent. Each
offspring particle then moves from the location of its parent by
making a random jump according to the step distribution $p (x,y)$; the
jumps are once again mutually independent. Consider the initial
condition which places a single particle at site $x\in \Gamma$, and
denote the corresponding probability measure on population evolutions
by $Q^{x}$.

\begin{proposition}\label{proposition:brw}
Under $Q^{x}$, the total number of particles in generation $n$ evolves
as a Galton-Watson process with offspring distribution
$\mathcal{Q}$. If the offspring distribution has mean $r\leq R$, then
under $Q^{x}$ the expected number of particles at location $y$ at time
$n$ is $r^{n}P^{x}\{X_{n}=y \}$, where under $P^{x}$ the process
$X_{n}$ is an ordinary random walk with step distribution
$\mathcal{P}$. Therefore, $G_{r} (x,y)$ is the mean total number of
particle visits to location $y$.
\end{proposition}

\begin{proof}
The first assertion follows easily from the definition of a
Galton-Watson process -- see \cite{athreya-ney} for the definition and
basic theory.  The second is easily proved by induction on $n$. The third
then follows from the formula \eqref{eq:green} for the Green's
function.
\end{proof}

There are similar interpretations of the restricted Green's function
$G_{r} (x,y;\Omega)$ and the first-passage generating function $F_{r}
(x,y)$. In particular, if particles of the branching random walk are
allowed to reproduce only in the region $\Omega$, then
$G_{r}(x,y;\Omega)$ is the mean number of particle visits to $y$ in
this modified branching random walk.

\section{Exponential decay of the Green's function}\label{sec:apriori}

\subsection{Hyperbolic geometry} In this section, we prove that the
Green's function of a symmetric random walk on a co-compact Fuchsian
group decays exponentially.  The proof is most conveniently
formulated using some basic ingredients of hyperbolic geometry (see,
e.g., \cite{katok}, chs.\  3--4).

For any co-compact Fuchsian group $\Gamma$ there is a
\emph{fundamental polygon} $\mathcal{F}$ for the action of $\Gamma$.
The closure of $\mathcal{F}$ is a compact, finite-sided polygon whose
sides are arcs of hyperbolic geodesics. The hyperbolic disk $\zz{D}$
is tiled by the images $x \mathcal{F}$ of $\mathcal{F}$, where $x\in
\Gamma $; distinct elements $x,y\in \Gamma$ correspond to tiles
$x\mathcal{F}$ and $y\mathcal{F}$ which intersect, if at all, in a
single side or vertex of each tile. Thus, the elements of $\Gamma$
are in bijective correspondence with the tiles of the tessellation,
or alternatively with the points $xO$ in the $\Gamma -$orbit of a
distinguished point $O\in \mathcal{F}^{\circ}$. The Cayley graph of
$\Gamma$ (relative to the standard generators) is gotten by putting
\emph{edges} between those \emph{vertices} $x,y\in \Gamma$ such that
the tiles $x\mathcal{F}$ and $y\mathcal{F}$ share a side. The
\emph{word metric} $d=d_{\Gamma}$ on $\Gamma$ is defined to be  the
(Cayley) graph distance.  The hyperbolic metric induces another
distance $d_{\zz{H}} (x,y)$ on $\Gamma$, defined to be the hyperbolic
distance between the points $xO$ and $yO$. The metrics $d$ and
$d_{\zz{H}}$ are both left-invariant.

The following fact is well known (see, e.g., \cite{cannon:2}, Theorem~1).

\begin{lemma}
\label{lemma:quasiIsometric} The restriction of the hyperbolic metric $d_{\zz{H}}$
to $\Gamma$ and the word metric $d$ on $\Gamma$ are equivalent.
\end{lemma}

\subsection{Barriers}

To prove that the Green's function decays exponentially fast, we will
construct \emph{barriers} along the hyperbolic geodesic between two
points $x O$ and $y O$ in such a way that the total weight $w_R$
attached to paths crossing any of those barriers is small. Recall that
a \emph{path} in $\Gamma$ is a sequence $x_{i}$ such that successive
jumps $x_{n}x_{n+1}^{-1}$ are all elements of $S$, the support of the
step distribution of the random walk. Recall also that $S$ is
symmetric and contained in the ball $B(1,C_{0})$ of radius $C_{0}$
centered at the group identity $1$. A \emph{halfplane} is a subset of
$\zz{D}$ (or its intersection with $\Gamma O$) whose boundary is a
doubly infinite hyperbolic geodesic.


\begin{definition}\label{definition:barrier}
A \emph{barrier} is a triple $(V,W,B)$ consisting of nonoverlapping
halfplanes $V,W \subset \zz{D}$ and a subset $B\subset \Gamma$
disjoint from $V\cup W$ such that every path in $\Gamma O$ from $V$ to
$W$ must pass through $B$, and
\[
	\max_{x\in V} \sum_{b\in B} G_R(x, b) \leq \frac{1}{2} .
\]
For distinct points $\xi ,\zeta\in \partial\zz{D}$, a barrier between
$\xi$ and $\zeta$ is a barrier $(V,W,B)$ such that $\xi$ and $\zeta $
are interior points of the arcs $\text{cl}(V)\cap \partial \zz{D}$
and $\text{cl}(W)\cap
\partial \zz{D}$, respectively. (We also say that such a barrier
$(V,W,B)$ \emph{separates} $\xi$ and $\zeta$.)
\end{definition}

Observe that if $(V,W,B)$ is a barrier then the distance between the
hyperbolic geodesics bounding $V$ and $W$ is positive: in particular,
these geodesics cannot have a common endpoint on $\partial
\zz{D}$. Moreover, the set $B$ must have limit points in each of the
two arcs of $\partial \zz{D}$ separating $\text{cl}(V)\cap \partial
\zz{D}$ and $\text{cl}(W)\cap \partial \zz{D}$. Also, if $(V,W,B)$ is
a barrier then (i) for any $g\in \Gamma$  the translate $(gV,gW,
gB)$ is a barrier, and (ii) if $V',W'$ are halfplanes contained in
$V,W$, respectively, then $(V',W',B)$ is a barrier.


\begin{theorem}
\label{theorem:barriersExist}
For any two points $\xi\neq \eta \in \partial \zz{D}$, there exists a barrier between
$\xi$ and $\eta$.
\end{theorem}

Before turning to the proof, we record an  easy
consequence.

\begin{corollary}\label{corollary:separationThreshold}
There exists $K<\infty$ such that for any two hyperbolic geodesics
$\alpha ,\beta $ at distance $K$ or greater there is a barrier
$(A,B,C)$ such that the halfplanes $A$ and $B$ are bounded by $\alpha$
and $\beta$, respectively.
\end{corollary}
\begin{proof}
Let $Q$ be the set of all pairs of distinct points $(\xi ,\zeta) \in
\partial \zz{D}$ such that the geodesic with endpoints $\xi$ and
$\zeta$ enters the interior of the fundamental polygon $\mathcal{F}$.
By Theorem~\ref{theorem:barriersExist}, any such pair $\xi ,\zeta$ is
separated by a barrier $(V,W,B)$; the halfplanes $V,W$ can be chosen
to have arbitrarily small (Euclidean) diameters, and so in particular
we can assume that neither $V$ nor $W$ intersects $\mathcal{F}$. If
$(\xi ,\zeta )\in Q$ are separated by the barrier $(V,W,B)$ then all
nearby pairs $(\xi ',\zeta ')$ are also separated by this barrier.
Since $Q$ has compact closure in $\partial \zz{D}\times
\partial \zz{D}$, it follows that there is a finite set of barriers
$(V_{i},W_{i},B_{i})$ and a constant $\delta >0$ such that (i) for any
pair $(\xi,\zeta)\in Q$ at least one of the barriers separates $\xi$
from $\zeta$, and (ii) for this barrier $(V_{i},W_{i},B_{i})$ the
Euclidean circles of radius $\delta$ centered at $\xi$ and $\zeta$
intersect $\zz{D}$ only inside the halfplanes $V_{i}$ and $W_{i}$,
respectively.  Furthermore, the barriers can be chosen so that none of
the halfplanes $V_{i},W_{i}$ intersects the fundamental polygon $\mathcal{F}$.

Suppose that $\alpha ,\beta$ are hyperbolic geodesics at distance $D$
greater than four times the diameter of $\mathcal{F}$. Let $\gamma$
be the geodesic segment of minimal hyperbolic length with endpoints
on $\alpha$ and $\beta$, and let $M$ be the midpoint of this segment.
Then there is an element $g\in \Gamma$ that maps the geodesic segment
$\gamma$ to a segment $g\gamma$ which crosses $\mathcal{F}$ in such a
way that the midpoint $gM$ lies in one of the tiles $h\mathcal{F}$
within distance $\text{\rm diam} (\mathcal{F})$ of $\mathcal{F}$. Let
$\xi,\zeta$ be the endpoints on $\partial \zz{D}$ of the doubly
infinite geodesic extension of $g\gamma$. By construction, the pair
$(\xi ,\zeta)$ is an element of $Q$, so $\xi$ and $\zeta$ are
separated by one of the barriers $(V_{i},W_{i},B_{i})$.  If the
distance $D$ between $\alpha$ and $\beta $ is sufficiently large, say
$D\geq K$, then the geodesics $g\alpha$ and $g\beta$ will be entirely
contained in the halfplanes $V_{i}$ and $W_{i}$, respectively. (This
is because the midpoint $gM$ of the connecting segment $g\gamma$ lies
within distance $\text{\rm diam} (\mathcal{F})$ of $\mathcal{F}$, so
both $g\alpha$ and $g\beta $ are at distance greater than
$D/2-2\text{\rm diam} (\mathcal{F})$ from $\mathcal{F}$, and
consequently, if $D$ is sufficiently large, must lie inside circles
of Euclidean radius $\delta$ centered at points on $\partial
\zz{D}$.) If $g\alpha$ and $g\beta$ are contained in $V_{i}$ and
$W_{i}$, then there is a barrier $(gA,gB,gC)$ such that the
halfplanes $gA$ and $gB$ are bounded by $g\alpha$ and $g\beta$.
Finally, since translates of barriers are barriers, it follows that
$( A,B,C)$ is a barrier with the desired property.
\end{proof}

The rest of this subsection is devoted to the proof of Theorem
\ref{theorem:barriersExist}.  The strategy is to construct the barrier
using typical trajectories of the random walk.

\begin{lemma}
There exist $C>0$ and $\varrho<1$ such that $P^1(G_R(1,X_n) \geq
\varrho^n) \leq C \varrho^n$.
\end{lemma}

\begin{proof}
Since $G_{r} (1,x)=F_{r} (1,x)G_{r} (1,1)$, it suffices to prove the
corresponding statement where the Green's function $G_{R} (1,X_{n})$
is replaced by the first-passage generating function $F_{R}
(1,X_{n})$. Now any path $\gamma$  of length $n$ from $1$ to a point
$x$ can be concatenated with any path $\gamma'$ from $x$ to $1$,
yielding a path from $1$ to itself. If the path $\gamma '$ does not
re-visit the point $x$ the splitting $(\gamma ,\gamma ')$ of the
concatenated path is a last-exit decomposition, so the mapping
$(\gamma ,\gamma')\mapsto \gamma \gamma'$ is injective.  Hence,
summing the weights of all such paths gives a lower bound for $G_{R}
(1,1)$: using the symmetry $F_{R} (x,1)=F_{R} (1,x)$, we obtain
\begin{equation*}
  G_R(1,1) \geq
  \sum_x  R^n P^{1} (X_{n}=x) F_R(1, x)
  = R^n E^1( F_R(X_n,1)).
\end{equation*}
Therefore,
\begin{equation*}
  P^1(F_R(X_n, 1) \geq \varrho^n)
  \leq \frac{1}{\varrho^n} E^1(F_R(X_n,1))
  \leq CR^{-n} \varrho^{-n},
\end{equation*}
and so the desired inequality holds with $\varrho =R^{-1/2}$.
\end{proof}

\begin{lemma}
\label{lemma:summable}
For almost all independent trajectories $X_0=1,X_1,\dots$ and
$Y_0=1, Y_1, \dots$ of the random walk,
  \begin{equation*}
  \sum_{m,n\in \zz{N}} G_R(X_m, Y_n) < \infty.
  \end{equation*}
\end{lemma}
\begin{proof}
For fixed $m$ and $n$, $Y_n^{-1}X_m$ is distributed as $X_{n+m}$, by
symmetry of the step distribution. Therefore,
  \begin{equation*}
  P( G_R(X_m, Y_n) \geq \varrho^{m+n})
  =P(G_R(Y_n^{-1} X_m, 1) \geq \varrho^{m+n})
  =P(G_R(X_{n+m}, 1) \geq \varrho^{m+n})
  \leq C \varrho^{n+m},
  \end{equation*}
by the previous lemma. Since this quantity is summable in $m$ and
$n$, Borel-Cantelli ensures that, almost surely, $G_R(X_m, Y_n) \leq
\varrho^{m+n}$ for all but finitely many pairs $(m,n)$.
\end{proof}

We will construct barriers by first constructing
\emph{pre-barriers} as defined below.

\begin{definition}
A \emph{pre-barrier} is a quadruple $(V,W,A,B)$ of pairwise disjoint
sets such that $V,W$ are non-overlapping halfplanes and $A,B$ are  subsets
of $\Gamma$ such that
\begin{enumerate}
\item [(i)] every path in $\Gamma$ from $V$ to $W$ must enter $A\cup
B$, with first entrance at a point $a\in A$ and last exit at a point
$b\in B$;
\item [(ii)] every path in $\Gamma$ from $V$ to $B$ must pass through
$A$; and
\item [(iii)] we have
  \begin{equation*}
  \sum_{a\in A, b\in B} G_R(a, b) < \infty.
  \end{equation*}
\end{enumerate}
A pre-barrier $(V,W,A,B)$ \emph{separates} points $\xi ,\zeta \in
\partial \zz{D}$ if $\xi$ and $\zeta $ are interior points of
the arcs $\text{cl}(V)\cap \partial \zz{D}$ and $\text{cl}(W)\cap
\partial \zz{D}$, respectively.
\end{definition}

\begin{lemma}
For any two distinct points $\xi, \eta\in \partial \zz{D}$, there
exists a pre-barrier separating $\xi$ and $\eta$.
\end{lemma}

\begin{proof}
Choose two disjoint compact subintervals $J$ and $K$ in one of the
connected components of $\partial \zz{D} - \{\xi, \eta\}$ such that
$J$ is closer to $\xi$ than $K$, and similarly choose $J'$, $K'$ in
the other connected component of $\partial \zz{D} - \{\xi,\eta\}$.

Let $X_n, X'_n, Y_n, Y'_n$ be independent copies of the random walk
starting from $1$, and let $\Omega'$ be the event that $X_n$
converges to $J$, $X'_n$ converges to $J'$, $Y_n$ converges to $K$
and $Y'_n$ converges to $K'$. The Poisson boundary of the random walk
is identified with the topological boundary $S^1$, and the limiting
measure has full support there since the group action is minimal on
the boundary. Therefore, $\Omega'$ has positive probability.
Moreover, Lemma \ref{lemma:summable} implies that almost surely $\sum
G_R(\tilde X_m, \tilde Y_n)<\infty$ for $\tilde X=X$ or $X'$ and
$\tilde Y=Y$ or $Y'$. Consequently, there exist two-sided infinite
paths $\{x_{n} \}_{n\in \zz{Z}}$ connecting $J$ with $J'$ and
$\{y_{n} \}_{n\in \zz{Z}}$ connecting $K$ with $K'$ such that
$\sum_{m,n\in \zz{Z}} G_R(x_m, y_n)<\infty$.   Since $J\cup J'$ is
disjoint from $K\cup K'$, the trajectories $x_{m}$ and $y_{n}$ are
disjoint for large $|m|, |n|$. Modifying these trajectories at
finitely many places (which does not change the validity of
$\sum_{m,n} G_R(x_m, y_n)<\infty$) and then replacing each point of
the trajectory by a ball of radius $2C_0$ (so that a trajectory of
the random walk can not jump past the newly constructed set), we
obtain the required pre-barrier.
\end{proof}

Theorem \ref{theorem:barriersExist} follows from the previous lemma
and the next lemma.

\begin{lemma}
If $(V,W,A,B)$ is a pre-barrier separating $\xi$ and $\eta$, then
there exist halfplanes $V'\subseteq V$ and $W'\subseteq W$ such that $(V',W',B)$ is a
barrier separating $\xi$ and $\eta$.
\end{lemma}

\begin{proof}
The only nontrivial point is the existence of a halfplane neighborhood $V'$ of
$\xi $ such that  $\max_{x\in V'}\sum_{b\in B}G_R(x,
b)\leq 1/2$.

Let $V'$ be a small neighborhood of $\xi$. Every path from $V'$ to
$B$ must go first through $A$.  Decomposing such a path according to
the first visited point in $A$, we find that for any $b\in B$,
\begin{align*}
    G_R(x, b) &\leq \sum_{a\in A} G_R(x, a) G_R(a, b)\\
    	   &\leq \sup_{a\in A} G_R(x, a) \cdot \sum_{a\in A, b\in B}
  G_R(a,b) \\
  &\leq C \sup_{a\in A} G_R(x, a).
\end{align*}
By Lemma~\ref{lemma:backscattering}, $G_R(1, x)$ tends uniformly to
$0$ when $|x|\to \infty$. Therefore, if $V'$ is small enough then
$\sup_{a\in A} G_R(x, a)$ is smaller than $1/(2C)$.
\end{proof}

\subsection{Exponential decay of the Green's function}\label{ssec:decay}

\begin{theorem}\label{theorem:expDecay}
The Green's function $G_R(x,y)$ of a symmetric, irreducible,
finite-range random walk on a co-compact Fuchsian group $\Gamma$,
evaluated at its radius of convergence $R$, decays exponentially in
the distance $d(x,y)$. In particular, there exist constants $C<\infty$ and
$\varrho<1$ such that for every $x\in \Gamma$,
\begin{equation*}
	G_R(1,x)\leq C\varrho^{|x|}.
\end{equation*}
\end{theorem}

To prove this theorem, we will show that disjoint barriers can be
placed consecutively along the hyperbolic geodesic from $1$ to $x$
(identified with $O$ and $x O$) in $\zz{D}$.

\begin{lemma}
\label{lemma:multipleBarriers}
There exists $L>0$ with the following property. If $x \in \Gamma$
satisfies $d(1,x) \geq nL$ for some integer $n\geq 2$, then there
exist $n$
barriers  $(V_{i},W_{i},B_{i})$
 such that
\begin{itemize}
\item [(i)] the sets $B_{i}$ are pairwise disjoint;
\item [(ii)] every path from $1$ to $x$  goes
    successively through $B_{0},B_1, B_2,\dots,\allowbreak B_{n-1}$;
\item [(iii)] we have $\sum_{b_0 \in B_0} G_R(1, b_0) \leq 1/2$; and
\item [(iv)] for any $i<n-1$ and any $b_i\in B_i$,
  \begin{equation}
  \label{eq:successiveBarriers}
  \sum_{b_{i+1}\in B_{i+1}} G_R(b_i, b_{i+1}) \leq 1/2.
  \end{equation}
\end{itemize}
\end{lemma}

\begin{proof}
This is an easy consequence of Corollary
\ref{corollary:separationThreshold} and
Lemma~\ref{lemma:quasiIsometric}. Let $\gamma$ be the hyperbolic
geodesic segment from $1$ to $x$; by Lemma~\ref{lemma:quasiIsometric},
if $L$ is sufficiently large then the hyperbolic length of $\gamma$ is
at least $nK$. Let $\beta_{0},\beta_{1},\dotsc \beta_{n}$ be the
hyperbolic geodesics that cross $\gamma$ orthogonally at the points
$1,y_{1},y_{2},\dotsc y_{n}$ at distances $0,K,2K,\dotsc ,nK$ from
the endpoint $1$. Corollary~\ref{corollary:separationThreshold}
implies that for each $0\leq i<n$ there is a barrier
$(V_{i},W_{i},B_{i})$ such that $V_{i}$ is bounded by $\beta_{i}$ and
$W_{i}$ by $\beta_{i+1}$.  These have all the desired properties.
\end{proof}

\begin{proof}[Proof of Theorem \ref{theorem:expDecay}] Let $x\in
\Gamma$ be a vertex such that $d (1,x)\in [nL,(n+1)L)$.  Decomposing
paths from $1$ to $x$ according to their first points of entry into
$B_0$, then $B_1$, and so on, we get
\begin{align*} G_R(1,x) & \leq
\sum_{b_0\in B_0,\dots, b_{n-1}\in B_{n-1}} G_R(1, b_0) G_R(b_0,
b_1)\dots G_R(b_{n-2}, b_{n-1}) G_R(b_{n-1}, x) \\& \leq 2^{-n+1}
\sup_{b_{n-1}\in B_{n-1}} G_R(b_{n-1}, x).
\end{align*}
Since $G_R$ tends to $0$ at infinity, it is uniformly bounded, and so
the exponential bound follows.
\end{proof}

\section{Ancona's inequalities}\label{sec:ancona}

In this section, we shall prove Ancona's inequalities
\eqref{eq:ancona} and extend them to \emph{relative} Green's
functions (Theorem~\ref{theorem:relativeAncona}). We will then show
how these inequalities are used to identify the Martin boundary and
to establish the H\"older continuity of the Martin kernel. The proofs
of the various Ancona inequalities will rely on the exponential decay
of the Green's function, which was proved in
section~\ref{sec:apriori}, and the hyperbolic character of the Cayley
graph. Recall that a finitely generated group $\Gamma$ (or
alternatively its Cayley graph) is \emph{hyperbolic} in the sense of
Gromov~\cite{gromov} if it satisfies the \emph{thin triangle
property}, that is, there exists a constant $\delta <\infty$ such
that for any geodesic triangle $T$ in $\Gamma$, each point on a side
of $T$ is within distance $\delta$ of at least one of the other two
sides. We then say that is $\delta$ is a \emph{Gromov constant} for
$\Gamma$. The minimal Gromov constant $\delta$ will in general depend
on the choice of the generating set; however, if $\Gamma$ is
hyperbolic with respect to one (finite) set of generators then it is
hyperbolic with respect to every finite set of generators.

\subsection{Relative Ancona's inequalities}

For technical reasons, we will need an extension of Ancona's
inequalities valid for \emph{relative}
Green's functions (as defined in subsection~\ref{ssesc:fpGF}). These
inequalities assert that, for any point $z$ on a geodesic segment
between two points $x$ and $y$, and for a suitable class of domains
$\Omega$, one has
\begin{equation}\label{eq:relAncona}
		   G_r(x, y;\Omega) \leq C G_r(x,z;\Omega) G_r(z,y;\Omega).
\end{equation}
The constant $C$ should be independent of $x,y,z$, of $\Omega$, and
of $r\in [1,R]$. The usual Ancona inequalities \eqref{eq:ancona}
correspond to the case $\Omega= \Gamma$. Clearly, inequality
\eqref{eq:relAncona} cannot hold for all domains $\Omega$, because in
general there might not be positive-probability paths from $x$ to $z$
in $\Omega$. Therefore, some restrictions on the set $\Omega$ are
necessary.

Relative versions of Ancona's inequalities play an important role in
\cite{ancona:annals} (see th.~1' there), \cite{ledrappier:review,
izumi}. In the latter two papers, such inequalities are
proved for any domain $\Omega$ that contains a neighborhood of fixed
size $C$ of the geodesic segment $[xy]$. We are only able to deal with
a smaller class of domains, but these will nevertheless be sufficient
for our purposes.

\begin{theorem}
\label{theorem:relativeAncona} For any symmetric, irreducible,
finite-range random walk on a co-compact Fuchsian group $\Gamma$ there
exist positive constants $C$ and $C'$ such that the following
statement holds.  For any geodesic segment $[xy]$ and any $z\in [xy]$,
and for any subset $\Omega \subset \Gamma$ with the property that for
each $w\in [xy]$ the ball of radius $C'+d (w,z)/2$ centered at $w$ is
contained in $\Omega$,
\begin{equation}\label{eq:relativeAncona}
  G_r(x,y; \Omega)\leq C G_r(x,z; \Omega) G_r(z,y; \Omega) \quad
  \text{for all}\;\; 1\leq r\leq R.
  \end{equation}
\end{theorem}

There is another natural class of domains $\Omega$ for which we can
establish relative Ancona's inequalities (under additional
assumptions on the random walk).

\begin{definition}\label{definition:convex}
A set $\Omega \subset \Gamma$ is \emph{convex} if for all
pairs $x,y\in \Omega$ there is a geodesic segment from $x$ to $y$
contained in $\Omega$.
\end{definition}

\begin{theorem}
\label{theorem:relativeAnconaConvex}
Consider an irreducible, symmetric random walk on a co-compact
Fuchsian group $\Gamma$, and assume that its step distribution gives
positive probability to any element of the generating set used to
define the word distance. There exists a constant $C>0$ such that,
for every convex set $\Omega$ containing the geodesic segment $[xy]$,
and for every point $z\in [xy]$, the Ancona inequality
\eqref{eq:relativeAncona} holds.
\end{theorem}

The converse inequality $G_r(x,y; \Omega)\geq C G_r(x,z; \Omega)
G_r(z,y; \Omega)$ is trivial (see Lemma~\ref{lemma:superadditivity}),
so the theorems really say that $G_r(x,y; \Omega)\asymp G_r(x,z;
\Omega) G_r(z,y; \Omega)$ (meaning that the ratio between the two
sides of this equation remains bounded away from $0$ and $\infty$).

The proof of the two  theorems will use the following lemma, which is an
easy consequence of the exponential decay of the Green's function and
the geometry of the hyperbolic disk.

\begin{lemma}
\label{lemma:superExponential}
There exist $C>0$ and $\alpha >0$ such that, for any geodesic segment
$[xy]$,  any $z\in [xy]$ and  any $k\geq 0$,
  \begin{equation*}
  G_R(x,y; B_{k}(z)^c) \leq C \exp \{-e^{\alpha k} \}.
  \end{equation*}
\end{lemma}

Here $B_{k}(z)^c$ is the complement of the ball $B_{k}(z)$ of radius
$k$ centered at $z$. Since the distance from $x$ to $y$ in the
complement of $B_{k}(z)$ grows exponentially with $k$, and the Green's
function itself decays exponentially, the double exponential bound in
the conclusion of the lemma should not be surprising.

\begin{proof}
It suffices to prove the lemma for large $k$. Since the Green's
function is translation invariant, we can assume without loss of
generality that $z=1$.  (Recall that $1\in \Gamma$ is identified with
the point $O\in \zz{D}$.)  Consider the hyperbolic circle
$\mathcal{C}$ of radius $k/C$ around $O$, where $C$ is sufficiently
large that this circle is contained in the ball $B_{k}(1)$ for the
Cayley graph distance. Since the hyperbolic length of $\mathcal{C}$
grows exponentially with $k$, one can set along this circle an
exponential number $e^{\alpha k}$ of geodesic rays $D_i$ starting from
$O$ and going to infinity, such that the minimum distance $d$ between
any successive rays $D_i$ and $D_{i+1}$ along $\mathcal{C}$ is
arbitrarily large (the exponential rate $\alpha >0$ will not depend
depend on $d$, but the minimum radius $k/C$ will.). Let $\tilde D_i$
be a thickening of $D_i$ so that a path making jumps of length at most
$C_0$ can not jump across $\tilde D_i$. If $d$ is large enough then
$\sum_{a_i \in \tilde D_i, a_{i+1}\in \tilde D_{i+1}} G_R(a_i,
a_{i+1}) \leq 1/2$ (this sum is dominated by a geometric series,
thanks to the exponential decay of the Green's function, and it has
arbitrarily small first term if $d$ is large). A path from $x$ to $y$
that avoids $B_{k}(1)$ has to go around the circle in one direction or
the other, and will therefore cross at least half the domains $\tilde
D_i$. As in the conclusion of the proof of Theorem
\ref{theorem:expDecay}, this yields $G_R(x,y; B_{k}(1)^c) \leq C 2^{-
e^{\alpha k}/2}$.
\end{proof}

\begin{proof}[Proof of Theorems
\ref{theorem:relativeAncona}--\ref{theorem:relativeAnconaConvex}] Let
$[xy]$ be a geodesic segment of some length $m$ (relative to the
Cayley graph distance) with endpoints $x,y\in \Gamma $, let $z \in
\Gamma $ be a point on $[xy]$, and let $\Omega$ be a domain
containing $[xy]$ satisfying the assumptions of
Theorem~\ref{theorem:relativeAncona}
or~\ref{theorem:relativeAnconaConvex}.

We first construct by induction points $x_n, y_n$ on the geodesic
segment $[xy]$ in such a way that for each $n$ the points
$x_{n},z,y_{n}$ occur in order on $[xy]$. Start by setting $x_0=x$
and $y_0=y$. At step $n$, if $z$ is in the left half of $[x_n y_n]$
then set $x_{n+1}=x_n$ and let $y_{n+1}$ be a point in $[x_n y_n]$
such that $|d(y_{n+1},y_n)-d(x_n,y_n)/4|\leq 1$. Similarly, if $z$ is
in the right half of $[x_{n}y_{n}]$ then define $y_{n+1}=y_n$ and let
$x_{n+1}$ be a point on $[x_{n}y_{n}]$ such that $|d(x_n,
x_{n+1})-d(x_n, y_n)/4|\leq 1$. The construction ends at the first
step $N$ such that $d(x_N, y_N)\leq A$, for some large $A>0$. By
construction, for all $n\leq N$,
\[
	 d(x_n,y_n) = (3/4)^n m +O(1) \quad \text{and} \quad z\in [x_n y_n],
\]
where the $O(1)$ term is bounded by $\sum_{k=0}^\infty (3/4)^k = 4$.

At each step of the construction, the discarded interval
($[y_{n+1}y_{n}]$ in the first case, $[x_{n}x_{n+1}]$  in the second) lies either
to the right or to the left of $z$ on the interval $[xy]$. In either
case, define $B_{n+1}$ to be the ball of radius $d (x_{n},y_{n})/100$
centered at the midpoint of the discarded interval. Observe that these
balls are pairwise disjoint. For each $n$, let $a_{n}$ and $b_{n}$ be
the midpoints of the \emph{last} discarded intervals to the left and
right of $z$ respectively (with the convention that if no intervals to
the left of $z$ have been discarded by step $n$ then $a_{n}=x$, and
similarly if no intervals to the right have been discarded then
$b_{n}=y$). By construction,
\[
	a_{n}\leq x_{n}\leq z \leq y_{n}\leq b_{n}
\]
in the natural ordering (left to right) on $[xy]$, and so the center
of the ball $B_{n+1}$ must lie on the segment
$[a_{n}b_{n}]$. Furthermore, by hyperbolicity, if $u_{n},v_{n}$ are
any points in the balls $B_{l},B_{r}$ closest to $z$ on the left and
right, respectively, centered at points in intervals removed by step
$n$, then any geodesic segment $[u_{n}v_{n}]$ connecting $u_{n}$ and
$v_{n}$ must pass within distance $2\delta$ of $z$, where $\delta$ is
a Gromov constant for the Cayley graph (provided the constant $A$ that
determines the termination point of the construction is sufficiently
large). In fact, the entire geodesic segment $[x_{n}y_{n}]$ must lie
within distance $2\delta$ of $[u_{n}v_{n}]$.

We now decompose $G_r(x,y;\Omega)$ by splitting random walk
trajectories at visits to the balls $B_i$. Begin with  $u_0=x$ and
$v_0=y$. If $z$ is in the first half of $[u_0 v_0]$, split paths at
the \emph{last} visit to $B_1$; this yields
\[
	G_r(u_0,v_0; \Omega) =  G_r(u_0,v_0; \Omega \cap B_1^c)
  + \sum_{v_1 \in B_1 \cap \Omega } G_r(u_0,v_1; \Omega) G_r(v_1, v_0; \Omega\cap B_1^c).
\]
If $z$ is in the second half of $[u_0 v_0]$, split paths at the \emph{first}
visit  to $B_1$; this gives
  \begin{equation*}
  G_r(u_0,v_0; \Omega) = G_r(u_0,v_0; \Omega \cap B_1^c)+ \sum_{u_1
  \in B_1 \cap \Omega }
      G_r(u_0,u_1; \Omega \cap B_1^c) G_r(u_1, v_0; \Omega).
  \end{equation*}
To get manageable formulas, we will introduce a more symmetric
notation. Let $H(u,v; B)= G_r(u,v; \Omega \cap B)$ if $u\not=v$, and
$1$ if $u=v$. In addition, write  $u_1=u_0$ in the first case (that
is, if  $z$ is in the first half of $[u_0 v_0]$), and write
$v_1=v_0$ in the second case. The formulas above become
  \begin{equation*}
  G_r(u_0, v_0; \Omega)=H(u_0, v_0; B_1^c)
    + \sum_{u_1, v_1} H(u_0, u_1; B_1^c) G_r(u_1, v_1; \Omega) H(v_1, v_0; B_1^c).
  \end{equation*}

This procedure can  be iterated. The factor
$G_{r}(u_{1},v_{1};\Omega)$ can be decomposed  by splitting random walk
trajectories at  visits to $B_{2}$; this leads to a second sum with an
inner factor $G_{r} (u_{2},v_{2};\Omega)$. This factor can again be
decomposed, and so on, whence  we obtain, by induction, for any
$k\leq N$,
  \begin{align}
  \label{eq:GRhorrible}
  \raisetag{-80pt}
  \begin{split}
  G_r(u_0, v_0; \Omega)
   =& \sum_{j=0}^{k-1} \sum_{\substack{u_1,\dots, u_j\\ v_1,\dots,
  v_j}}H(u_j, v_j; B_{j+1}^c) \prod_{i=0}^{j-1}\left\{ H (u_{i},u_{i+1};B_{i+1}^{c})H
  (v_{i+1},v_{i};B_{i+1}^{c}) \right\}
  \\&
  +
  \sum_{\substack{u_1,\dots, u_k\\ v_1,\dots, v_k}}G_r(u_k, v_k;
  \Omega)
     \prod_{i=0}^{k-1} \left\{ H (u_{i},u_{i+1};B_{i+1}^{c})H
     (v_{i+1},v_{i};B_{i+1}^{c})\right\}.
  \end{split}
  \end{align}
For each index $i$ in the products, either $u_{i}=u_{i+1}$ or
$v_{i}=v_{i+1}$, and the corresponding $H-$factor is $1$. In the
first case, $v_{i+1}\in B_{i+1}\cap \Omega $; in the second case,
$u_{i+1}\in B_{i+1}\cap \Omega$. Thus, for each $i$ the points
$u_{i},v_{i}$ must lie in the balls $B_{l},B_{r}$ centered at the
midpoints of the nearest discarded (by step $i$) intervals to the
left and right of $z$. As noted earlier, this implies that any
geodesic segment $[u_{i}v_{i}]$ must pass within distance $2\delta$
of $z$. Moreover, since the center of the ball $B_{i+1}$ lies on
$[x_{i}y_{i}]$, the ball contains  a ball with center on
$[u_{i}v_{i}]$ of radius $(m (3/4)^{i}-4)/100-2\delta$.  Therefore,
by Lemma~\ref{lemma:superExponential}, for suitable constants
$C,\beta
>0$,
\[
	H (u_{i},v_{i};B_{i+1}^{c}) \leq C\exp \{-e^{\beta  m (3/4)^{i}}\}.
\]
On the other hand, under the hypotheses of
Theorem~\ref{theorem:relativeAncona} (for sufficiently large $C'$), a
fixed size neighborhood of a geodesic segment $[u_{i}v_{i}]$ from
$u_{i}$ to $v_{i}$ is contained in $\Omega$, while under the
hypotheses of Theorem~\ref{theorem:relativeAnconaConvex} the geodesic
segment $[u_{i}v_{i}]$ is contained in $\Omega$ by convexity. It
follows that, in both situations (and assuming the random walk in
nearest-neighbor in the second situation), one has
$G_{r}(u_{i},v_{i};\Omega)\geq p_{\min}^{d (u_{i},v_{i})}$ for all
$r\geq 1$, for some $p_{\min}>0$. Consequently,
  \begin{equation*}
  H(u_i, v_i; B_{i+1}^c) \leq \lambda_i G_r(u_i, v_i; \Omega),
  \end{equation*}
where $\lambda_i$ is superexponentially small in terms of $m
(3/4)^i$. The $j$th term of the first sum in \eqref{eq:GRhorrible} is
therefore bounded by
  \begin{equation*}
  \lambda_j \sum_{\substack{u_1,\dots, u_j\\ v_1,\dots, v_j}}
   G_r(u_j, v_j; \Omega)\prod_{i=0}^{j-1} \left\{ H
   (u_{i},u_{i+1};B_{i+1}^{c})H (v_{i+1},v_{i};B_{i+1}^{c})\right\}.
  \end{equation*}
This matches the last line of \eqref{eq:GRhorrible} (with $j$ replacing
$k$). Since every term in \eqref{eq:GRhorrible} is nonnegative, it
follows that the   $j$th term of the first sum in \eqref{eq:GRhorrible}
is bounded above by $\lambda_j G_r(u_0, v_0; \Omega)$.

Now take $k=N$ (i.e., the last step of the construction), and consider
the final sum in \eqref{eq:GRhorrible}. Since the points $u_N$ and
$v_N$ are within distance $A$ of $z$, we have $G_r(u_N, v_N; \Omega)
\leq C_{A} G_r(u_N, z; \Omega) G_r(z, v_N; \Omega)$, for a suitable
constant $C_{A}<\infty$. When this upper bound is substituted for the
factor $G_r(u_N, v_N; \Omega)$, the final sum factors (one factor
involving only the $u_i$s, the other only the $v_i$s). Reversing the
path decomposition (i.e., gluing trajectories instead of splitting
them) shows that the two resulting factors are bounded respectively by
$G_r(u_0, z; \Omega)$ and $G_r(z, v_0; \Omega)$. Thus, we have
 \begin{equation*}
G_r(u_0, v_0; \Omega) \leq
\left(\sum_{j=0}^{N-1} \lambda_j \right) G_r(u_0, v_0; \Omega) + C_{A}
G_r(u_0, z; \Omega) G_r(z, v_0; \Omega).
\end{equation*}
If $A$ is
large enough, the sum $\sum \lambda_j$ will be smaller than $\leq
1/2$. Hence, for all large $A$,
 \begin{equation*}
G_r(u_0, v_0; \Omega) \leq 2 C_{A} G_r(u_0, z;
\Omega) G_r(z, v_0; \Omega).  \qedhere
\end{equation*}
\end{proof}

\subsection{H\"older continuity of the Green's function}
\label{subsec:fellow_traveling}

In this paragraph, we explain how the controls on the Martin boundary
given by Theorem~\ref{theorem:holderMartinKernel} follow from the
relative Ancona's inequalities of the previous paragraph. The
strategy of the proof is essentially the same as that introduced by
Anderson and Schoen~\cite{anderson-schoen} for a similar purpose in a
slightly different setting, and the details follow \cite{izumi}
closely. (The theorem could also be proved by an adaptation of the
argument used in \cite{series}.)

\begin{definition}\label{definition:shadows}	
Let $[xy]$ and $[x'y']$ be geodesic segments in the Cayley graph of
$\Gamma$, and let $\varepsilon >0$. We say that the segment $[x'y']$
shadows (more precisely, $\varepsilon -$shadows) the segment $[xy]$
if every point of $[xy]$ lies within distance $\varepsilon$ of
$[x'y']$. If both $[x''y'']$ and $[x'y']$ $(2\delta)-$shadow $[xy]$,
where $\delta$ is a Gromov constant for the Cayley graph, then we say
that they are \emph{fellow-traveling} along $[xy]$.
\end{definition}

\begin{theorem}
\label{theorem:holderAncona}
There exist constants $C>0$ and $\varrho<1$ such that for any geodesic
segment $[x_{0}y_{0}]$, if $[xy]$ and $[x'y']$ are fellow-traveling
along $[x_{0}y_{0}]$, then
  \begin{equation*}
  \left|\frac{G_r(x,y)/G_r(x',y)}{G_r(x,y')/G_r(x',y')} - 1\right|\leq C \varrho^k,
  \end{equation*}
for all $r\in [1,R]$, where $k$ is the length of $[x_{0}y_{0}]$.
\end{theorem}

A direct application of Ancona's inequalities imply that
  \begin{equation}
  \label{equation:anconaBasic}
  \frac{G_r(x,y)}{G_r(x',y)} \asymp \frac{G_r(x,x_0)}{G_r(x',x_0)} \asymp \frac{G_r(x,y')}{G_r(x',y')}.
  \end{equation}
The theorem is a quantitative strengthening of this estimate, showing
that the ratio between those quantities not only stays bounded, but
tends exponentially fast to $1$ when $k$ tends to infinity.

In particular,  take $x'=1$ and let  $y_n, y_m$ be points at distances
$n,m$ from $x'$ along a geodesic ray converging towards a point $\zeta \in
\partial \Gamma$. For any fixed $x\in \Gamma$ the geodesic segments
from $x$ or $x'$ to $y_n$ or $y_m$  are fellow
traveling along a geodesic segment of length at least $\min(m,n)-2d (x,x')$. Therefore,
  \begin{equation*}
  \left|\frac{G_r(x,y_n)/G_r(1,y_n)}{G_r(x,y_m)/G_r(1,y_m)} -
  1\right|\leq C_x \varrho^{\min(m,n)}
  \end{equation*}
for a constant $C_{x}<\infty$ depending on $x$.
This shows that the sequence $G_r(x,y_n)/G_r(1,y_n)$ is Cauchy,
therefore convergent to a limit $K_r(x, \zeta)$. To prove
Theorem~\ref{theorem:martinBoundary}, one should additionally show
that the functions $K_r(x,\zeta)$ are minimal, and that for
$\zeta\not=\zeta'$ one has $K_r(x,\zeta) \not= K_r(x,\zeta')$. Ancona
proved this for $r<R$ in~\cite{ancona:annals}, and his proofs also work
for $r=R$ once the Ancona inequalities are established. Finally,
letting $m$ tend to infinity, one gets $|G_r(x,y_n)/G_r(1,y_n) -
K_r(x,\zeta)|\leq C_x \varrho^n K_r(x,\zeta)$, which proves
Theorem~\ref{theorem:holderMartinKernel}.

\begin{proof}[Proof of Theorem \ref{theorem:holderAncona}] Let $L\gg
2\delta $ be a large constant, and consider a geodesic segment
$[x_0y_0]$ of length $k\gg L$. For each $1\leq i\leq \ell:=k/(3L)$,
define $\Omega_{i}$ to be the set of all $z\in \Gamma$ such that
every geodesic segment from $y_{0}$ to $z$ passes through the ball of
radius $2\delta$ centered at the point $z_{i}$ at distance $3Li$ from
$y_{0}$ along $[x_{0}y_{0}]$.  The sequence of domains $\Omega_{i}$
is decreasing in $i$ (by an easy application of the thin triangle
property). Moreover, if the geodesic segments $[xy]$ and $[x'y']$ are
fellow-traveling along $[x_{0}y_{0}]$, then both $x,x'\in
\Omega_{\ell}$ and $y,y'\in \Omega_{1}^{c}$.

Consider the function $u(z)=G_r(z,y)/G_r(x_0,y)$, which is
$r$-harmonic on $\Omega_1$ and normalized by $u(x_0)=1$. Starting with
$u_{1}=u$, we will inductively construct a sequence of $r-$harmonic
functions $u_i$ on $\Omega_i$, with $u_{i-1}=u_i+\varphi_i$, in such a
way that $\varphi_i$ does not depend on the initial normalized
harmonic function $u$,  and so that $u_{i-1}\geq
\varphi_i \geq \varepsilon u_{i-1}$ on $\Omega_i$, for some
$\varepsilon>0$. Let us first show how this gives the conclusion of
the theorem.

As $\varphi_i\geq \varepsilon u_{i-1}$, we have $u_i \leq
(1-\varepsilon)u_{i-1}$, hence $u_\ell \leq (1-\varepsilon)^{\ell-1}
u_1$. Applying the same construction to $v(z)=G_r(z,y')/G_r(x_0,
y')$, we obtain $v=\sum \varphi_i+v_\ell$ (for the same functions
$\varphi_i$), which gives on $\Omega_\ell$ the estimate
$|u-v|=|u_\ell-v_\ell| \leq C (1-\varepsilon)^\ell (u+v)$. Since
$u\asymp v$ by \eqref{equation:anconaBasic}, we get $|u/v- 1|\leq C
(1-\varepsilon)^\ell$ on $\Omega_\ell$, which is the desired
inequality.

We now describe the construction of $\varphi_i$. Assume that
$u_1,\dots, u_i$ have been defined. By harmonicity, for any $z\in
\Omega_{i+1}$,
  \begin{equation*}
  u_i(z)=\sum_{w\in \Omega_i^c} G_r(z,w;\Omega_i) u_i(w).
  \end{equation*}
Define $\Lambda_{i}$ to be the set of all $z\in \Omega_{i}$ such that
$(x_0|z)_{y_0} \in [3Li + L -2, 3Li+L+2]$, where
$(x|y)_z=(d(x,z)+d(y,z)-d(x,y))/2$ is the Gromov product, which
measures the distance along which two geodesic segments from $z$ to
$x$ and from $z$ to $y$ are fellow traveling.
This set is contained in $\Omega_i$, but is bounded away from $\Omega_{i+1}$,
and any trajectory from $\Omega_{i+1}$ to the complement of
$\Omega_i$ has to cross $\Lambda_i$. Splitting a trajectory from
$z\in \Omega_{i+1}$ to $w\in \Omega_i^c$ according to its last visit
to $\Lambda_i$, we get $ G_r(z,w;\Omega_i) = \sum_{w'\in \Lambda_i}
G_r(z,w';\Omega_i) G_r(w', w; \Omega_i \cap \Lambda_i^c)$.

We are now in a position to estimate $G_r(z,w';\Omega_i)$ using
Ancona's inequalities. Indeed, the geodesic segment from $z$ to $w'$
passes within $2\delta$ of the point $z^{*}_i$ at distance $3Li+2L$
from $y_0$ on $[x_0y_0]$, by hyperbolicity. Moreover, the domain
$\Omega_i$ satisfies the assumptions of
Theorem~\ref{theorem:relativeAncona} (this readily follows from a
tree approximation). Hence, $G_r(z,w';\Omega_i) \leq C G_r(z,
z^{*}_i;\Omega_i) G_r(z^{*}_i, w';\Omega_i)$, and so
  \begin{align*}
  u_i(z)
  &
  =\sum_{w\in \Omega_i^c}\sum_{w'\in \Lambda_i} G_r(z,w';\Omega_i)
  G_r(w',w; \Omega_i \cap \Lambda_i^c)u_i(w)
  \\&
  \leq  C G_r(z,z^{*}_i;\Omega_i)
  \sum_{w\in \Omega_i^c}\sum_{w'\in \Lambda_i} G_r(z^{*}_i,w';\Omega_i)
  G_r(w',w; \Omega_i \cap \Lambda_i^c)u_i(w)
  \\&
  = C G_r(z,z^{*}_i;\Omega_i) u_i(z^{*}_i).
  \end{align*}
Replacing Ancona's inequality by the trivial bound
$G_r(z,w';\Omega_i) \geq  C' G_r(z, z^{*}_i;\Omega_i) G_r(z^{*}_i,
w';\Omega_i)$, we also get a lower bound in the last equation. Hence,
$u_i(z)\asymp G_r(z,z^{*}_i;\Omega_i) u_i(^{*}z_i)$ on
$\Omega_{i+1}$. Using this estimate for $z=x_0$, we obtain
$u_i(z)/u_i(x_0) \asymp G_r(z,z^{*}_i;\Omega_i)/G_r(x_0,
z^{*}_i;\Omega_i)$. In particular, if $c$ is small enough, the
function $\varphi_{i+1}(z)=c u_i(x_0)
G_r(z,z^{*}_i;\Omega_i)/G_r(x_0, z^{*}_i;\Omega_i)$ satisfies
$\varepsilon u_i \leq \varphi_{i+1} \leq u_i$ on $\Omega_{i+1}$.
Moreover, this function only depends on $u_i$ through the value of
$u_i(x_0)$. By induction, it only depends on $u(x_0)$. Since we have
normalized $u$ so that $u(x_0)=1$, this shows that $\varphi_{i+1}$ is
independent of the initial function $u$.
\end{proof}

\section{Automatic structure}\label{sec:cannon}

\subsection{Strongly Markov groups and
hyperbolicity}\label{ssec:strongMarkov}

A finitely generated group $\Gamma $ is said to be \emph{strongly
Markov} (fortement Markov -- see \cite{ghys-deLaHarpe}) if for each
finite, symmetric generating set $A$ there exists a finite directed
graph $\mathcal{A}= (V,E,s_{*})$ with distinguished vertex $s_{*}$
(``start'') and a labeling $\alpha :E \rightarrow A$ of edges by
generators that meets the following specifications. A \emph{path} in
the graph is a sequence of edges $e_0,\dots,e_{m-1}$ such that the
endpoint of $e_i$ is the starting point of $e_{i+1}$. Let
\begin{equation*}
	\mathcal{P}:=\{\text{finite paths in $\mathcal{A}$ starting at $s_{*}$} \},
\end{equation*}
and for each path $\gamma =e_{0}e_{1}\dotsb e_{m-1}$, denote by
\begin{align*}
    \alpha (\gamma)&= \text{path in $G^{\Gamma }$ through $1,\alpha
	(e_{0}),\alpha (e_{0})\alpha (e_{1}),\dotsc$},
    \quad \text{and}\\
	\alpha_{*} (\gamma)&=\alpha (e_{0})\alpha (e_{1}) \dotsb \alpha(e_{m-1}),\text{ the right endpoint of }\alpha (\gamma).
\end{align*}

\begin{definition}\label{definition:automaticStructure}
The labeled automaton $(\mathcal{A},\alpha )$ is a strongly Markov
automatic structure for $\Gamma$ if:
\begin{enumerate}
\item [(A)] No edge $e\in E$ ends at $s_{*}$.
\item [(B)] Every vertex $v\in V$ is accessible from the start
    state $s_{*}$.
\item [(C)]  For every path $\gamma$, the path $\alpha (\gamma)$
    is a geodesic path in $G^{\Gamma}$.
\item [(D)]  The endpoint mapping
    $\alpha_{*}:\mathcal{P}\rightarrow \Gamma $ induced by
    $\alpha$ is a bijection of $\mathcal{P}$ onto $\Gamma$.
\end{enumerate}
\end{definition}

\begin{theorem}\label{theorem:cannon}
Every word hyperbolic group is strongly Markov.
\end{theorem}

See \cite{ghys-deLaHarpe}, Ch.~9, Th.~13. The result is essentially
due to Cannon (at least in a more restricted form) --- see
\cite{cannon:1},  \cite{cannon:3} --- and in important
special cases (co-compact Fuchsian groups) to Series
\cite{series}. Henceforth, we will call the directed graph
$\mathcal{A}= (V,E,s_{*})$ the \emph{Cannon automaton} (despite the
fact that it is not quite the same automaton as constructed in
\cite{cannon:1}).

Properties (C)-(D) of Definition~\ref{definition:automaticStructure}
imply that for each $x\in \Gamma$ there is a \emph{unique} geodesic
segment in the Cayley graph from the group identity $1$ to $x$ that
is the image of a path in the automaton. We shall denote this
distinguished geodesic segment by $L (1,x)$.

\subsection{Automatic structures for the surface
groups}\label{ssec:autoSurface}

The existence of an automatic structure will be used to connect the
behavior of the Green's function at infinity to the theory of Gibbs
states and Ruelle operators (see \cite{bowen}, ch.~1). For these
arguments, the detailed structure of the automaton will not be
important (except for those aspects discussed in
sec.~\ref{ssec:restrictions} below).  Nevertheless, we note here that
an automatic structure $\mathcal{A}$ for the \emph{surface group}
$\Gamma_{g}$ is easily constructed. Let $A=A_{g}=\{a_{i}^{\pm},
b_{i}^{\pm}\}$ be the standard generating set, with the generators
satisfying the basic relation
\begin{equation}\label{eq:surfaceRelations}
	\prod_{i=1}^{g} a_{i}b_{i}a_{i}^{-1}b_{i}^{-1}=1.
\end{equation}
Define the set $V$ of vertices for the automaton to be the set of all
reduced words in the generators of length $\leq 2g$, with $s_{*}=$
the empty word. Directed edges are set according to the following
rules:
\begin{itemize}
\item [(A)] If a (reduced) word $w'$ is obtained by adding a single
letter $x$ to the end of word $w$, then draw an edge $e (w,w')$ from
$w$ to $w'$, and label it with the letter $x$.
\item [(B)] If a word $w'$ of maximal length $2g$ is
obtained from another word $w$ of length $2g$ by deleting the first
letter and adding a new letter $x$ to the end, then draw an edge $e
(w,w')$ from $w$ to $w'$ with label $x$ \emph{unless} the word $wx$
constitutes the first $2g+1$ letters of a cyclic permutation of the
basic relation \eqref{eq:surfaceRelations}.
\end{itemize}
That properties (C)--(D) of
Definition~\ref{definition:automaticStructure} are satisfied follows
from Dehn's algorithm. The words of maximal length $2g$ are the
\emph{recurrent vertices} of this automaton, while the words of length
$<2g$ are the \emph{transient vertices} (see
sec.~\ref{ssec:restrictions} below for the definitions). It is
easily verified that for any vertex $w$ and any \emph{recurrent} vertex
$w'$, there is a path in the automaton from $w$ to $w'$.

\subsection{Recurrent and transient vertices}\label{ssec:restrictions}

Let $\mathcal{A}$ be a Cannon automaton for the group $\Gamma$ with
vertex set $V$ and (directed) edge set $E$.  Call an edge $e\in E$
\emph{recurrent} if there is a path in $\mathcal{A}$ of length $\geq
2$ that begins and ends with $e$; otherwise, call it
\emph{transient}. Denote by $\mathcal{A}_{R}$ the restriction of the
digraph $\mathcal{A}$ to the set $\mathcal{R}$ of recurrent edges.
For certain hyperbolic groups --- among them the co-compact Fuchsian
groups --- the automatic structure can be chosen so that the digraph
$\mathcal{A}_{R}$ is \emph{strongly connected} (see \cite{series}),
i.e., for any two recurrent edges $e$ and $e'$ there is a path from
$e$ to $e'$. Henceforth we restrict attention to word-hyperbolic
groups with this property:

\begin{assumption}\label{assumption:irreducibility}
The automatic structure can be chosen so that the digraph
$\mathcal{A}_{R}$ is strongly connected.
\end{assumption}

\begin{assumption}\label{assumption:mixing}
The incidence matrix of the digraph $\mathcal{A}_{R}$ is aperiodic.
\end{assumption}

Both assumptions hold for any co-compact Fuchsian group.
Assumption~\ref{assumption:mixing} is for ease of exposition only ---
the results and arguments below can be modified to account for any
periodicities that might arise if the assumption were to fail.
Assumption~\ref{assumption:irreducibility}, however, is essentially
important.

\subsection{Symbolic dynamics}\label{ssec:symbolicDynamics}

We shall assume for the remainder of the paper that the automaton
$\mathcal{A}$ has been chosen so as to satisfy Assumptions
\ref{assumption:irreducibility} and \ref{assumption:mixing}.

Set
\begin{align*}
	\Sigma &=
	 \{\text{semi-infinite  paths in} \;\mathcal{A} \},\\
	\Sigma^{n}&=\{\text{paths of length $n$ in}\; \mathcal{A}\},\\
	\Sigma^{*}&=\cup_{n=0}^{\infty}\Sigma^{n},\\
    \bSigma &= \Sigma \cup \Sigma^*.
\end{align*}
By convention, there is a single path of length $0$, the empty path,
that we denote by $\emptyset$. In some circumstances, it is useful to
identify finite paths with semi-infinite paths in an automaton with an
additional ``cemetery'' state, or with doubly-infinite paths in an
automaton with two additional states (``embryo'' and ``cemetery'').
 This point of view makes it possible to apply
directly in our setting results that are formulated in the literature
only for semi-infinite paths. However, we stick to the notation with
finite paths since it makes the correspondence with geodesic segments
in the group (see below) more transparent.

We will also need bilateral versions of these sets, that we will
denote with a subscript $\Z$. For instance, $\bSigma_{\Z}$ is the
set of (finite or infinite) bilateral paths in $\mathcal{A}$.
Equivalently, it is the set of sequences $(\omega_n)_{n\in \Z}$ where
$\omega_n$ is an edge of $\mathcal{A}$ for $n$ in some interval of
$\Z$ (with admissible transition from $\omega_n$ to $\omega_{n+1}$),
and $\omega_n$ is empty for $n$ outside of this interval. Let $\sigma
$ be the forward shift operator on $\bSigma$ and $\bSigma_{\Z}$. The
spaces $\bSigma$ and $\bSigma_{\Z}$ are given metrics in the usual
way, that is,
\begin{equation*}
	d (\omega ,\omega')=2^{-n (\omega ,\omega ')}
\end{equation*}
where $n (\omega ,\omega ')$ is  the maximum integer $n$ such that
$\omega_{i}=\omega_{i}'$ for all $|i|<n$. With the topology induced
by $d$ the space $\Sigma$ is a Cantor set, $\Sigma $ is the set of
accumulation points of $\Sigma^{*}$, and $\bSigma$ and $\bSigma_{\Z}$
are compact.  Observe that, relative to the metrics $d$,
H\"older-continuous, real-valued functions on $\Sigma^*$ extend by
continuity to H\"older-continuous functions on $\bSigma$, and then pull
back to H\"older-continuous functions on $\bSigma_{\Z}$.

Each $\omega \in \Sigma$ projects via the edge-labeling map $\alpha $
to a geodesic ray in $G^{\Gamma}$ starting at the vertex $1$ (more
precisely, the sequence of finite prefixes of $\omega$ project to the
vertices along a geodesic ray). Each geodesic ray in $G^{\Gamma}$
must converge in the Gromov topology to a point of $\partial \Gamma$,
so $\alpha$ induces on $\Sigma$ a mapping $\alpha_*$ to $\partial
\Gamma$. By construction, this mapping is H\"older continuous relative
to any visual metric on $\partial \Gamma$. Moreover, because each $\zeta \in
\partial \Gamma$ is the limit of a geodesic ray starting at the vertex
$1$, the induced mapping $\alpha_{*}$ is surjective.

In a somewhat different way, the edge-labeling map $\alpha$
determines a map from the space $\bSigma_{\Z}$ to the set of
two-sided (finite or infinite) geodesics in $G^{\Gamma}$ that pass
through the vertex $1$. This map is defined as follows: if $\omega
\in \bSigma_{\Z}$ then the image of $\omega$ is the two-sided
geodesic that passes through
  \begin{equation}
  \label{eq:alpha_bilat}
	\dotsc , \alpha (\omega_{-1}^{-1})\alpha (\omega_{-2}^{-1}),
	\alpha (\omega_{-1}^{-1}),1,\alpha (\omega_{0}),\alpha
	(\omega_{0}) \alpha (\omega_{1}), \dotsc ,
  \end{equation}
equivalently, it is the concatenation of the geodesic rays starting at
$1$ that are obtained by reading successive steps from the sequences
\[
	\omega_{0}\omega_{1}\omega_{2}\dotsb
	\quad \text{and}\quad
	\omega_{-1}^{-1}\omega_{-2}^{-1}\omega_{-3}^{-1}\dotsb ,
\]
respectively. When $\omega$ is bi-infinite, each of these geodesic
rays converges to a point of $\partial \Gamma$, so $\alpha$ induces a
mapping from $\Sigma_{\Z}$ into $\partial \Gamma \times \partial
\Gamma$. This mapping is neither injective nor surjective, but it is
H\"older-continuous.

Let $E_*$ be the set of edges originating from $s_*$, and let
$\Sigma^m(E_*)$ be the set of sequences of length $m$ in $\Sigma^m$
with $\omega_0\in E_*$. By definition of the Cannon automaton, the
mapping $\alpha_*$ induces a bijection between $\Sigma^m(E_*)$ and
the sphere $S_{m}$ of radius $m$ in $G^{\Gamma}$.

\begin{corollary}\label{corollary:sphereGrowth}
Let $\zeta$ be the spectral radius of the incidence matrix of the
digraph $\mathcal{A}$. If Assumptions~\ref{assumption:irreducibility}
and \ref{assumption:mixing} hold, then $\zeta >1$, and there exists
$C>0$ such that
\begin{equation*}
	|S_{m}|\sim C\zeta^{m} \quad \text{as} \;\; m \rightarrow \infty .
\end{equation*}
\end{corollary}
\begin{proof}
This follows directly from the Perron-Frobenius theorem, with the
exception of the assertion that the spectral radius $\zeta >1$. That
$\zeta >1$ follows from the fact that the group $\Gamma$ is
nonelementary. Since $\Gamma $ is nonelementary, it is nonamenable,
and so its Cayley graph has positive Cheeger constant; this implies
that $|S_m|$ grows exponentially with $m$.
\end{proof}

\begin{corollary}\label{corollary:positiveEntropy}
The shift $(\Sigma, \sigma)$ has positive topological entropy.
\end{corollary}
\begin{proof}
This follows from the exponential growth of the group, cf.\ %
Corollary~\ref{corollary:sphereGrowth}.
\end{proof}

\section{Thermodynamic formalism}\label{sec:thermo}

\emph{Assume throughout this section and
sections~\ref{sec:pressureEvaluation}--\ref{sec:criticalExp} that
the group $\Gamma$ is co-compact Fuchsian, and that the random walk is
symmetric and finite-range.}

\subsection{The potential functions
\texorpdfstring{$\varphi_{r}$}{phi r}}\label{ssec:potentialFunctions}

The machinery of thermodynamic formalism and Gibbs states developed
in \cite{bowen} applies to H\"older continuous functions on $\Sigma$
(or on $\bSigma$). To make use of this machinery, we will lift the
Green's function and the Martin kernel from $\overline \Gamma$ to the
sequence space $\bSigma$. For this the results of
Theorem~\ref{theorem:holderMartinKernel} and
Theorem~\ref{theorem:holderAncona} are crucial, as they  ensure that
those lifts are H\"older-continuous. The lift is defined as follows.
For $\omega \in \Sigma^*$,  set
\begin{equation}\label{eq:phi-r-definition-a}
	\varphi_{r} (\omega):=\log \frac{G_{r}(1,\alpha_{*}
	(\omega))}{G_r(1, \alpha_*(\sigma \omega))}.
\end{equation}
If $\omega$ is not the empty path, one can also write
\[
  \varphi_{r} (\omega)
  =\log \frac{G_{r}(1,\alpha_{*} (\omega))}{G_r(\alpha_*(\omega_0), \alpha_*(\omega))}.
\]
Therefore, Theorem~\ref{theorem:holderAncona} shows that, if two
paths $\omega$ and $\omega'$ coincide up to time $n$, then
$|\varphi_r(\omega)-\varphi_r(\omega')|\leq C \varrho^n$, for some
$\varrho<1$. By definition of the distance on $\Sigma^*$, this means
that $\varphi_r$ is H\"older-continuous. In particular, it extends to a
H\"older-continuous function (that we still denote by $\varphi_r$) on
$\bSigma$. On $\Sigma$, it is given by
\begin{equation}\label{eq:defCocycle}
  \varphi_{r} (\omega )
  = \log \frac{K_{r} (1,\alpha_{*}
  (\omega))}{K_{r} (\alpha_*(\omega_0), \alpha_{*}(\omega))}
  = - \log K_{r} (\alpha_*(\omega_0), \alpha_{*}(\omega)).
\end{equation}
The mapping $r\mapsto \varphi_r$ is clearly continuous at every point
of $\Sigma^*$, and all the functions $\varphi_r$ are uniformly
H\"older-continuous for some fixed exponent. Therefore, $r\mapsto
\varphi_r$ is also continuous for the sup norm, and it follows that
it is continuous for the H\"older topology respective to any H\"older
exponent strictly less than the initial one.

By construction, if $\omega$ is of length $n$,
\begin{equation}\label{eq:cocycle}
	G_r(1, \alpha_*(\omega)) = G_r(1,1) \exp(S_n \varphi_r(\omega))
\end{equation}
where (in Bowen's notation \cite{bowen})
\begin{equation*}
	 S_{n}\varphi :=\sum_{j=0}^{n-1}\varphi \circ \sigma^{j}.
\end{equation*}
(Unfortunately, the notation $S_{n}\varphi$ conflicts with the
notation $S_{m}$ for the sphere of radius $m$ in $\Gamma$; however,
both notations are standard, and the meaning should be clear in the
following by context.)

\subsection{Gibbs states: background}\label{sec:gibbs}

According to a fundamental theorem of ergodic theory (cf.\
\cite{bowen}, Th.~1.2 and sec.~1.4), for each H\"older continuous
function on a topologically mixing subshift of finite type, there is a
unique Gibbs state for this potential. Unfortunately, the subshift of
finite type induced by a Cannon automaton is not topologically mixing,
since the edges originating from $s_*$ are always transient (and for
hyperbolic groups in general, there can also be terminal edges, i.e.,
edges where a path can not be continued). However, under
Assumptions~\ref{assumption:irreducibility} and
\ref{assumption:mixing}, the recurrent part of the graph is
topologically mixing. Therefore, the existence of Gibbs states and the
corresponding Ruelle operator theory will generalize to our setting.

Consider a finite directed graph whose recurrent part is connected
and aperiodic, let $\bSigma$ be the set of finite or infinite paths
in this graph, and let $\sigma:\bSigma \to \bSigma$ be the left
shift. The assumption that the recurrent part of the graph is
connected and aperiodic implies that the restriction of $\sigma$ to
the set of infinite paths in the recurrent set is topologically
mixing. Let $\HH$ be the space of real-valued H\"older-continuous
functions on $\bSigma$ (for some fixed H\"older exponent). Let
$\mathcal{R}$ be the set of recurrent edges in the graph,
$\mathcal{R}^+$ the set of edges that can be reached from a recurrent
edge, and $\mathcal{R}^-$ the set of edges from which a recurrent
edge can be reached.

\begin{theorem}
\label{thm:RPF}
For any potential $\varphi\in \HH$,  define an operator
$\RL_\varphi$ acting on continuous functions $f:\bSigma \rightarrow \zz{R}$ by
  \begin{equation*}
  \RL_\varphi f(\omega)=\sum_{\sigma(\omega')=\omega} e^{\varphi(\omega')} f(\omega'),
  \end{equation*}
where if $\omega$ is the empty path the sum is restricted to the
preimages $\omega'$ of positive length. There exist a real number
$\Press(\varphi)$ (the \emph{pressure} of $\varphi$), a number
$\varepsilon>0$, a H\"older-continuous function $h_\varphi:\bSigma\to
\zz{R}^+$ and a probability measure $\nu_\varphi$ on $\bSigma$ such
that, for any $f\in \HH$, the following asymptotics hold in $\HH$:
  \begin{equation}
  \label{RL_asymtptotics}
  \RL_\varphi^n f = e^{n \Press(\varphi)} \left(\int f \, d\nu_\varphi\right) h_\varphi
  +O(e^{-\varepsilon n} e^{n \Press(\varphi)}).
  \end{equation}
The support of the function $h_\varphi$ is the set of sequences whose
elements all belong to $\mathcal{R}^+$, and $h_\varphi$ is bounded
away from zero there. The support of the measure $\nu_\varphi$ is
the set of infinite sequences whose elements all belong to
$\mathcal{R}^-$.

The measure $\mu_\varphi=h_\varphi\nu_\varphi$ is the \emph{Gibbs
measure} associated to the potential $\varphi$: it is a
$\sigma$-invariant probability
measure supported by the recurrent part $\Sigma^{\mathcal{R}}$ of
$\Sigma$, and it satisfies, for any
$\omega=(\omega_n)\in \Sigma^{\mathcal{R}}$,
  \begin{equation}
  \label{eq:Gibbs}
  C_1 \leq \frac{\mu_\varphi[\omega_0,\dots,\omega_{n-1}]}
  {e^{S_n \varphi(\omega) - n \Press(\varphi)}}
  \leq C_2,
  \end{equation}
where $C_1, C_2>0$ are two constants and the cylinder
$[\omega_0,\dots,\omega_{n-1}]$ is the set of sequences
$\omega'=(\omega'_0,\omega'_1,\dots)$ with $\omega'_i=\omega_i$ for
$0\leq i\leq n-1$.

Finally, all the quantities in the statement of the theorem (i.e.,
$\Press(\varphi)$, $\varepsilon$, $h_\varphi$, $\nu_\varphi$, $C_1$,
$C_2$, $\mu_\varphi$ and the implicit constant in the $O$--term in
\eqref{RL_asymtptotics}) vary continuously with $\varphi\in \HH$.
\end{theorem}

When the subshift of finite type is topologically mixing, this
theorem is proved in \cite{bowen}. Since the arguments there are
easily adapted to obtain the above version, we will only sketch a
proof, emphasizing the arguments that differ from those of \cite{bowen}.

\begin{proof}
Standard arguments using Lasota-Yorke estimates (see for instance
\cite{parry-pollicott} or \cite{baladi}) show that $\RL_\varphi$ has
a spectral gap on $\HH$: denoting by $e^{\Pr(\varphi)}$ the spectral
radius of $\RL_\varphi$, this operator has finitely many eigenvalues
of modulus $e^{\Pr(\varphi)}$, and the rest of its spectrum is
contained in a disk of strictly smaller radius. Using the positivity
of $e^{\varphi}$ and the fact that the recurrent part of $\Sigma$ is
topologically mixing, one can then prove that there is a unique
eigenvalue of maximal modulus, and that it is simple. The asymptotics
\eqref{RL_asymtptotics} follows. The eigenfunction and eigenmeasure
$h_\varphi$ and $\nu_\varphi$ satisfy respectively $\RL_\varphi
h_\varphi = e^{\Pr(\varphi)}h_\varphi$ and $\RL_\varphi^* \nu_\varphi
= e^{\Pr(\varphi)} \nu_\varphi$.

Consider next the support of $h_\varphi$. The results
in \cite{bowen} imply that $h_\varphi$ is positive, and bounded from
below, on the recurrent part $\Sigma^{\mathcal{R}}$ of $\Sigma$.
Since $h_\varphi$ is H\"older continuous, this implies that, if $n$ is
large enough, then $h_\varphi$ is also positive on elements of
$\bSigma$ of length at least $n$ whose first $n$ symbols are in
$\mathcal{R}$. Consider now a sequence $\omega$ whose symbols all
belong to $\mathcal{R}^+$. There exists a sequence $\alpha$ beginning
by $n$ symbols in $\mathcal{R}$ such that $\alpha \omega$ is a
possible path in the automaton. Therefore,
  \begin{equation*}
  h_\varphi(\omega)=e^{-n\Pr(\varphi)} \sum_{\sigma^n(\eta)=\omega} e^{S_n \varphi(\eta)} h_\varphi(\eta)
  \geq e^{-n\Pr(\varphi)} e^{S_n \varphi(\alpha\omega)} h_\varphi(\alpha\omega)>0.
  \end{equation*}
On the other hand, if $\omega$ contains a symbol not belonging to
$\mathcal{R}^+$, then $\omega$ has no preimage under $\sigma^n$ if
there is no path of length $n$ in the transient part of the
automaton. It follows that $\RL_\varphi^n h_\varphi(\omega)=0$, hence
$h_\varphi(\omega)=0$. This shows that the support of $h_\varphi$ is
exactly those sequences with all symbols in $\mathcal{R}^+$. Since
this set is compact, $h_\varphi$ is bounded from below there.

If $f$ is a continuous function, then
  \begin{equation*}
  \nu_\varphi(f)=e^{-n\Pr(\varphi)} \nu_\varphi(\RL_\varphi^n f).
  \end{equation*}
Since $\RL_\varphi^n f$ only depends on the values of $f$ on paths of
length at least $n$, this shows that $\nu_\varphi$ has no atom on
paths of finite length. Let us now take $f=1_{[C]}$ the
characteristic function of a cylinder $[C]$ of length $n$. Since
$\RL_\varphi^n 1_{[C]} (\omega)=e^{S_n \phi(C \omega)}$ if the
concatenation $C\omega$ is an admissible sequence, and $0$ otherwise,
we deduce that $\nu_\varphi[C]=0$ if $C$ can not be extended. If a
cylinder contains a symbol not in $\mathcal{R}^-$, it is a union of
cylinders that can not be extended, and has therefore $0$ measure. On
the other hand, if $C$ only contains symbols in $\mathcal{R}^-$, then
$[C]$ contains a cylinder $[C']$ of some length $m$ that can be
followed by a symbol $\omega_0$ in $\mathcal{R}$. Since $\RL^m
1_{[C']}$ is bounded from below on $[\omega_0]$, we get
$\nu_\varphi[C] \geq c \nu_\varphi[\omega_0]$, which is nonzero since
$\nu$ has full support in the recurrent part of $\Sigma$, by
\cite{bowen}. This shows that the support of $\nu_\varphi$ is exactly
the set of infinite paths whose symbols all belong to
$\mathcal{R}^-$.

The claims on the supports of $h_\varphi$ and $\nu_\varphi$ show that
the probability measure $\mu_\varphi=h_\varphi \nu_\varphi$ is
supported on the recurrent part of $\Sigma$. It coincides there with
the Gibbs measure constructed  in \cite{bowen}. Hence, \eqref{eq:Gibbs}
follows.

Finally, all the quantities in the statement of the theorem are
constructed from the spectral theory of the operator $\RL_\varphi$.
It then follows by standard arguments in regular perturbation theory
that they all vary continuously with $\varphi$ in the H\"older
topology.
\end{proof}

\subsection{Gibbs states and Green's function on
spheres}\label{ssec:greenSpheres}
Henceforth we denotes by $\mu_{r}$ the Gibbs measure on $\bSigma $
corresponding to the potential $\varphi_{r}$ defined by equations
\eqref{eq:phi-r-definition-a} and \eqref{eq:defCocycle}. Let
$\lambda_{r,m}$  be the probability measure on the sphere
$S_{m}\subset \Gamma $ with density proportional to $G_{r} (1,x)^{2}$,
that is, such that
\begin{equation}\label{eq:lambda-def}
	\lambda_{r,m} (x)=\frac{G_{r} (1,x)^{2}}{\sum_{y\in
	S_{m}}G_{r} (1,y)^{2}} \quad \text{for all} \;\; x\in S_{m}.
\end{equation}
Recall that $S_{m}$ is in one-to-one correspondence with the paths of
length $m$ in the automaton $\mathcal{A}$ that begin at $s_{*}$; in
particular, each $x\in S_{m}$ corresponds uniquely to a path $\omega
$ of length $m$ whose first step $\omega_0$ belongs to the set $E_*$ of edges
originating from $s_*$. Thus, for each $m\geq 1$, the probability
measure $\lambda_{r,m}$ on $S_{m}$ pulls back to a probability
measure on $\Sigma^{m}(E_*)\subset \bSigma$, which we also denote by
$\lambda_{r,m}$. This measure has density proportional to
\[
	 G_{r} (1,\alpha_{*} (\omega))^{2}= G_{r} (1,1)^{2}\exp\{2S_{m}\varphi_{r} (\omega ) \}  ,
\]
where $\omega \in \Sigma^{m}(E_*)$.

\begin{proposition}\label{proposition:absolutelyCont}
For each $r\in [1,R]$, the measures $\lambda_{r,m}$ on $\bSigma$
converge weakly as $m \rightarrow \infty$ to a probability measure
$\lambda_{r}$ on $\bSigma$, and  this convergence holds uniformly
in $r$, in the following sense: if $f:\bSigma\to \zz{R}$ is continuous, then
\begin{equation}\label{eq:uniformWeakConvergence}
	\lim_{m \rightarrow \infty} \int f\,d\lambda_{r,m} =
	\int f \,d\lambda_{r},
\end{equation}
uniformly for $r\in [1,R]$.
Furthermore, there exist constants $0<C=C (r;2)<\infty$ (depending
continuously on $r$) such that the normalizing constants in
\eqref{eq:lambda-def} satisfy
\begin{equation}\label{eq:sphereAsymptotics}
	\sum_{x\in S_{m}} G_{r} (1,x)^{2}
	\sim  C \exp \bigg\{m \Pr (2 \varphi_{r}) \bigg\}
\end{equation}
as $m \rightarrow \infty$.
\end{proposition}
\begin{proof}
We first prove \eqref{eq:sphereAsymptotics}. Let $\RL_r$ be the
Ruelle operator associated to the potential $2\varphi_r$. Denoting by
$\emptyset$ the path of length $0$ in $\bSigma$, and by
$1_{E_*}:\bSigma\to \zz{R}$ the function equal to $1$ on paths
originating from $s_*$ and $0$ otherwise, we have
  \begin{equation*}
  \RL_r^m 1_{E_*}(\emptyset) = \sum e^{2S_m \varphi_r(\omega)} = \sum G_r(1, \alpha_*(\omega))^2/G_r(1,1)^2,
  \end{equation*}
where the sum is over all paths $\omega$ of length $m$ originating
from $s_*$. Since $\alpha_*$ induces a bijection between such words
and $S_m$, we get
  \begin{equation*}
  \sum_{x\in S_{m}} G_{r} (1,x)^{2} = G_r(1,1)^2 \RL_r^m 1_{E_*}(\emptyset).
  \end{equation*}
By the Ruelle--Perron--Frobenius Theorem~\ref{thm:RPF}, this is
asymptotic to
  \begin{equation*}
  G_r(1,1)^2 e^{m \Pr(2\varphi_r)} \left(\int 1_{E_*} \, d\nu_r\right)
  h_r(\emptyset),
  \end{equation*}
where $\nu_r$ and $h_r$ are the eigenmeasure and eigenfunction of
$\RL_r$. Since $(\int 1_{E_*} \, d\nu_r) h_r(\emptyset)>0$ by
Theorem~\ref{thm:RPF}, we obtain \eqref{eq:sphereAsymptotics}.

We now turn to $\lambda_{r,m}$. This quantity can also be expressed in
terms of the transfer operator, as follows:
  \begin{equation}
  \label{eq:deflambdarm}
  \int f\, d\lambda_{r,m} = \RL_r^m (1_{E_*}f)(\emptyset)/\RL_r^m (1_{E_*})(\emptyset).
  \end{equation}
It follows again from Theorem~\ref{thm:RPF} that, if $f$ is
H\"older-continuous, then $\int f\, d\lambda_{r,m}$ converges to $\int
1_{E_*}f\, d\nu_r / \int 1_{E_*}\, d\nu_r$. Moreover, the convergence
is uniform for $r\in [1,R]$. If $f$ is merely continuous, it can be
uniformly approximated by a H\"older-continuous function, and the same
result follows. The limiting measure $\lambda_r$ is the normalized
restriction of $\nu_r$ to the paths starting from $s_*$.
\end{proof}

\begin{note}\label{note:theta}
Virtually the same argument shows that for any $\theta \in \zz{R}$, as
$m \rightarrow \infty$,
\[
	\sum_{x\in S_{m}} G_{r} (1,x)^{\theta }
	\sim  C \exp \bigg\{m \Pr (\theta  \varphi_{r}) \bigg\}.
\]
The result \eqref{eq:sphereAsymptotics} implies that $\Pr
(2\varphi_{r})< 0$ for all $r<R$ (see Lemma~\ref{corollary:belowR} below), and
Note~\ref{note:greenOnSphere} implies that $\Pr (\varphi_{r})>0$ for
all $r\in (1,R]$.  Since $\Pr (\theta \varphi_{r})$ varies
continuously with $\theta$, it follows that for each $r\in (1,R]$
there exists $\theta \in (1,2]$ such that $\Pr (\theta
\varphi_{r})=0$.  It can also be shown that the convergence of the
sums is uniform in $r$ for $r\in [1,R]$.
\end{note}

\begin{proposition}\label{corollary:ergodic}
Let $g:\bSigma \rightarrow \zz{R}$ be any H\"older-continuous function.
Then for each $\delta>0$
\begin{equation}\label{eq:ergodicTheorem}
		\lim_{m \rightarrow \infty}
	\lambda_{r,m}\left\{
	\omega \in \Sigma^{m}(E_*)\,:\, \Bigl \lvert
	     \frac{1}{m}\sum_{j=0}^{m-1} g\circ \sigma^{j} (\omega)
	     -\int g\,d\mu_{r}\Bigr \rvert >\delta
	 \right\} =0,
\end{equation}
and the convergence is uniform in $r\in [1,R]$.
\end{proposition}

\begin{proof}
To prove the convergence \eqref{eq:ergodicTheorem} for any particular
$r\in [1,R]$ it suffices, by Chebychev's inequality, to prove that the
variance of the average converges to zero as $m \rightarrow \infty$.
Replacing $g$ by $g-\int g\, d\mu_r$ and $\varphi_r$ by $\varphi_r
-\Pr(\varphi_r)$, we may assume that $\int g\,d\mu_r=0$ and
$\Pr(\varphi_r)=0$. We then have
  \begin{equation*}
  \int \left(\sum_{j=0}^{m-1} g\circ \sigma^{j}\right)^2 \,d\lambda_{r,m}
  = \sum_{j=0}^{m-1} \int (g\circ \sigma^j)^2 \,d\lambda_{r,m} + 2\sum_{j<k} \int g\circ \sigma^j \cdot g\circ \sigma^k\, d\lambda_{r,m}.
  \end{equation*}
The first sum is bounded by $m\xnorm{g}_{\infty }^{2}$. In the second
sum, for each $j<k$, we have, by \eqref{eq:deflambdarm},
\begin{align*}
  \int g\circ \sigma^j \cdot g\circ \sigma^k\, d\lambda_{r,m}
  &= \RL_r^m(1_{E_*}g\circ \sigma^j \cdot g\circ \sigma^k)(\emptyset)
  / \RL_r^m 1_{E_*}(\emptyset) \\
  &=\RL_r^{m-k}(g\RL_r^{k-j}(g\RL_r^j 1_{E_*}))(\emptyset) / \RL_r^m 1_{E_*}(\emptyset).
  \end{align*}
(The second equation follows from the identity $\RL(u\cdot v\circ \sigma)=\RL(u)\cdot v$.)
Theorem~\ref{thm:RPF} implies that  $\RL_r^n 1_{E_*} =
\nu_r[E_*] h_r + O(e^{-\varepsilon n})$, so the denominator $\RL_r^m
1_{E_*}(\emptyset)$ converges as $m \rightarrow \infty$ to a positive
constant $\nu_r[E_*] h_r (\emptyset)$, and
  \begin{equation*}
  \RL_r^{k-j}(g\RL_r^j 1_{E_*}))
  =\nu_r[E_*] \RL_r^{k-j}(g h_r) + O(e^{-\varepsilon j})
  =\nu_r[E_*] \left(\int g h_r \,d\nu_r\right) h_r+O(e^{-\varepsilon (k-j)}) + O(e^{-\varepsilon j}).
  \end{equation*}
Since $\int g h_r\, d\nu_r=\int g\, d\mu_r=0$, this is
$O(e^{-\varepsilon (k-j)}) + O(e^{-\varepsilon j})$. Summing over $j<
k< m$, we obtain the bound
  \begin{equation*}
  C \sum_{j< k< m} (e^{-\varepsilon j}+e^{-(k-j)})
  \leq C m.
  \end{equation*}
This shows  that the variance of the sum $S_{m}g$ is $O (m)$ as $m
\rightarrow \infty$, and so \eqref{eq:ergodicTheorem} follows.

Finally, all of the estimates obtained in the argument are uniform in
$r$, since all spectral data coming from Theorem~\ref{thm:RPF} are
already uniform.
\end{proof}

The ergodic average in \eqref{eq:ergodicTheorem} is expressed as an
average over the orbit of a path in the Cannon automaton, but it
readily translates to an equivalent statement for ergodic averages
along the geodesic segment $L=L (1,x)$, as we now explain. Consider a
function $f$, defined on the set of geodesic segments through $1$ in
the Cayley graph $G^\Gamma$. We say that it is \emph{H\"older
continuous} if $|f(L)-f(L')|\leq C\varrho^n$ for some $\varrho<1$
whenever two geodesics $L$ and $L'$ coincide on the ball of radius
$n$ around $1$. We put $f(L)=0$ if $L$ is a geodesic segment not
containing $1$. We have defined in~\eqref{eq:alpha_bilat} a map
$\alpha$ associating to a (finite) bilateral path in $\Sigma^{*}_{\Z}$
a geodesic segment in $G^\Gamma$. Therefore, $f\circ \alpha$ is a
function on $\Sigma^{*}_{\Z}$. It is H\"older continuous, and extends to
a H\"older continuous function on $\bSigma_{\Z}$.

There is a natural reference measure on the bilateral shift
$\bSigma_{\Z}$: since the Gibbs measure $\mu_r$ constructed on
$\Sigma$ in Theorem~\ref{thm:RPF} is shift--invariant, it extends to
a measure (denoted $\mu^{\Z}_r$) on $\Sigma_{\Z}$.

\begin{corollary}\label{corollary:ergodicCorollary}
Let $f$ be a H\"older continuous function on the space of geodesic
segments through $1$ in $G^\Gamma$. Then for each $\delta >0$,
\begin{equation}\label{eq:erg2}
			\lim_{m \rightarrow \infty}
	\lambda_{r,m}\left\{
	x\in S_{m}\,:\, \Bigr \lvert
	     m^{-1}\sum_{z\in  L (1,x)} f(z^{-1}L(1,x))
	     -\int f\circ \alpha\,d\mu^{\Z}_{r}\Bigr \rvert >\delta
	 \right\} =0,
\end{equation}
and the convergence is uniform in $r\in [1,R]$.
\end{corollary}
\begin{proof}
Let $\tilde f = f\circ \alpha$. If the function $\tilde f$ only
depends on the positive coordinates of a symbolic sequence, this
statement directly reduces to Corollary~\ref{corollary:ergodic} since
$S_m$ is in bijection with the set of paths of length $m$ starting
from $s_*$. By a theorem of Livsits (\cite{bowen}, Lemma 1.6 --- one
may either check that the proof still applies in our setting allowing
finite paths, or add cemeteries to reduce our situation to the
classical setting) any H\"older continuous function $\tilde f$ on
$\bSigma_{\Z}$ is \emph{cohomologous} to a H\"older continuous function
$g$ that depends only on the forward coordinates, i.e., there exists
a H\"older continuous function $u$ such that $g=\tilde f+u-u\circ
\sigma$. Since it is equivalent to have \eqref{eq:erg2} for $\tilde
f$ or $\tilde f+u-u\circ \sigma$, the general case follows.
\end{proof}

\section{Evaluation of the pressure at
\texorpdfstring{$r=R$}{r=R}}\label{sec:pressureEvaluation}

Proposition~\ref{proposition:absolutelyCont} implies that the sums
$\sum_{y\in S_{m}}G_{R} (1,y)^{2}$ grow or decay sharply
exponentially at exponential rate $\Pr(2\varphi_{R})$. Consequently,
to prove the relation \eqref{eq:backscatterA} of
Theorem~\ref{theorem:2} it suffices to prove that this rate is $0$.

\begin{proposition}\label{proposition:pressureEqualsZero}
$\Pr(2\varphi_{R})=0$.
\end{proposition}

The second assertion \eqref{eq:lp} of Theorem~\ref{theorem:2} also
follows from Proposition~\ref{proposition:pressureEqualsZero},  by
the main result of \cite{lalley:renewal}. (If it could be shown that
the cocycle $\varphi_{R}$ defined by \eqref{eq:defCocycle} above is
\emph{nonlattice} in the sense of \cite{lalley:renewal}, then the
result \eqref{eq:lp} could be strengthened from $\asymp$ to $\sim$.)

The remainder of this section is devoted to the proof of
Proposition~\ref{proposition:pressureEqualsZero}.  The first step,
that $\Pr(2\varphi_{R})\leq 0$, is a consequence of the differential
equations \eqref{eq:GPrime}.  These imply the following.

\begin{lemma}\label{corollary:belowR}
For every $r<R$,
\begin{align}\notag
	\Pr (2\varphi_{r})&<0, \quad \text{and so}\\
\label{eq:atR}
	\Pr (2\varphi_{R})&\leq 0.
\end{align}
\end{lemma}

\begin{proof}
For $r<R$ the Green's function $G_{r} (1,1)$ is analytic in $r$, so
its derivative must be finite. Thus, by
Proposition~\ref{proposition:GPrime}, the sum $\sum_{x\in
\Gamma}G_{r}(1,x)^{2}$ is finite. (The last term $r^{-1}G_r (1,1)$ in
equation \eqref{eq:GPrime} remains bounded as $r \rightarrow R-$
because $G_{R} (1,1,)<\infty$.)
Proposition~\ref{proposition:absolutelyCont}  therefore implies that
$\Pr (2\varphi_{r})$ must be negative.  Since $\Pr (\varphi)$ varies
continuously in $\varphi$, relative to the H\"older norm,
\eqref{eq:atR} follows.
\end{proof}

\begin{proof}
[Proof of Proposition~\ref{proposition:pressureEqualsZero}.]  To
complete the proof it  suffices, by the preceding lemma, to show that
$\Pr (2\varphi_{R})$ cannot be negative. In view of
Proposition~\ref{proposition:absolutelyCont}, this is equivalent to
showing that $\sum_{x\in S_{m}}G_{R} (1,x)^{2}$ cannot decay
exponentially in $m$.  This will be accomplished by proving that
exponential decay of $\sum_{x\in S_{m}}G_{R} (1,x)^{2}$ would force
\begin{equation}\label{eq:subcriticality}
	G_{r} (1,1)<\infty  \quad \text{for some} \; r>R,
\end{equation}
which is impossible since $R$ is the radius of convergence of the
Green's function.

To prove \eqref{eq:subcriticality}, we will use the branching random
walk interpretation of the Green's function discussed in
sec.~\ref{ssec:brw}.\footnote{Logically this is unnecessary --- the
argument has an equivalent formulation in terms of weighted paths,
using \eqref{eq:greenByPath} --- but the branching random walk
interpretation seems more natural.}  Recall that a branching random
walk on the Cayley graph $G^{\Gamma}$ is specified by an offspring
distribution $\mathcal{Q}$; assume for definiteness that this is the
Poisson distribution with mean $r>0$. At each step, particles first
produce offspring particles according to this distribution,
independently, and then each of these particles jumps to a randomly
chosen neighboring vertex. If the mean of the offspring distribution
is $r>0$, and if the branching random walk is initiated by a single
particle at the root $1$, then the mean number of particles located at
vertex $x$ at time $n\geq 1$ is $r^{n}P^{1}\{X_{n}=x \}$. Thus, in
particular, $G_{r} (1,1)$ equals the expected total  number of particle
visits to the root vertex $1$. The strategy is to show that if
 $\sum_{x\in S_{m}}G_{R} (1,x)^{2}$ decays exponentially in $m$,
then for some $r>R$ the branching random walk remains
\emph{subcritical}, that is, the expected total number of particle
visits to $1$ is finite.

Recall that the Poisson distribution with mean $r>R$ is the
convolution of  Poisson distributions with means $R$ and $\varepsilon
:=r-R$, that is, the result of adding independent random variables
$U,V$ with distributions Poisson-$R$ and Poisson-$\varepsilon$ is a
random variable $U+V$ with distribution Poisson-$r$. Thus, each
reproduction step in the branching random walk can be done by making
independent draws $U,V$ from the Poisson-$R$ and Poisson-$\varepsilon$
distributions. Use these independent draws to assign \emph{colors}
$k=0,1,2,\dotsc$ to the particles according to the following rules:

\begin{enumerate}
\item [(a)] The ancestral particle at vertex $1$ has color $k=0$.
\item [(b)]  Any offspring resulting from a $U-$draw has the
same color as its parent.
\item [(c)] Any offspring resulting from a $V-$draw has color equal to
$1+$the color of its parent.
\end{enumerate}

\begin{lemma}\label{lemma:colors}
For each $k=0,1,2,\dotsc$, the expected number of visits to the vertex
$y$ by particles of color $k$ is
\begin{equation}\label{eq:colors}
	v_{k} (y) =\varepsilon^{k} \sum_{x_{1},x_{2},\dotsc x_{k}\in
	\Gamma} G_{R} (1,x_{1})\left( \prod_{i=1}^{k-1} G_{R}
	(x_{i},x_{i+1})\right) G_{R} (x_{k},y)	.
\end{equation}
\end{lemma}

\begin{proof}
By induction on $k$. First, particles of color $k=0$ reproduce and
move according to the rules of a branching random walk with offspring
distribution Poisson-$R$, so the expected number of visits to vertex
$y$ by particles of color $k=0$ is $G_{R} (1,y)$, by
Proposition~\ref{proposition:brw}. This proves \eqref{eq:colors} in
the case $k=0$. Second, assume that the assertion is true for color
$k\geq 0$, and consider the production of particles of color
$k+1$. Such particles are produced only by particles of color $k$ or
color $k+1$. Call a particle a \emph{pioneer} if its color is
different from that of its parent, that is, if it results from a
$V-$draw. Each pioneer of color $k+1$ engenders its own branching
random walk of  descendants with color $k+1$;  the offspring distribution for
this branching random walk is the Poisson-$R$ distribution. Thus, for
a pioneer born at site $z\in \Gamma$, the expected number of visits to
$y$ by its color--$( k+1)$ descendants is $G_{R} (z,y)$. Every particle of color
$k+1$ belongs to the progeny of one and only one pioneer;
consequently, the expected number of visits to $y$ by particles of
color $k+1$ is
\[
	\sum_{z\in \Gamma }u_{k+1} (z) G_{R} (z,y),
\]
where $u_{k+1} (z)$ is the expected number of pioneers of color $k+1$
born at site $z$ during the evolution of the  branching process. But
since  pioneers of color $k+1$ must be children of parents of color
$k$, and since for any particle  the expected number of children of
different color is $\varepsilon$, it follows that
\[
	u_{k+1} (z) =\varepsilon v_{k} (z).
\]
Hence, formula \eqref{eq:colors} for $k+1$ follows by the induction hypothesis.
\end{proof}

Recall that our objective is to show that if $\sum_{x\in S_{m}}G_{R}
(1,x)^{2}$ decays exponentially in $m$ then
$G_{r} (1,1)<\infty$ for some $r=R+\varepsilon >R$. The branching
random walk construction exhibits $G_{r} (1,1)$ as the expected total
number of particle visits to the root vertex $1$, and this is the sum
over $k\geq 0$ of the expected number $v_{k} (1)$ of visits by
particles of color $k$. Thus, to complete the proof of
Proposition~\ref{proposition:pressureEqualsZero} it suffices, by
Lemma~\ref{lemma:colors}, to show that for some $\varepsilon >0$,
\begin{equation*}
	\sum_{k=0}^{\infty} \varepsilon^{k}\sum_{x_{1},x_{2},\dotsc x_{k}\in
	\Gamma}G_{R} (1,x_{1})\left( \prod_{i=1}^{k-1} G_{R} (x_{i},x_{i+1}) \right)
	G_{R} (x_{k},1)<\infty .
\end{equation*}
This follows directly from the next lemma.
\end{proof}

\begin{lemma}\label{lemma:snapback}
Assume that Ancona's inequalities \eqref{eq:ancona} hold at the
spectral radius $R$ with a constant $C_{R}<\infty$. If the sum  $\sum_{x\in
S_{m}}G_{R} (1,x)^{2}$ decays exponentially in $m$, then there exist
constants $\delta >0$ and $C,\varrho  <\infty$ such that for
every $k\geq 1$,
\begin{equation}\label{eq:snapback}
	\sum_{x_{1},x_{2},\dotsc x_{k}\in
	\Gamma} G_{R} (1,x_{1})\left( \prod_{i=1}^{k-1} G_{R} (x_{i},x_{i+1}) \right)
	(1+\delta)^{|x_{k}|}G_{R} (x_{k},1)\leq C\varrho ^{k}.
\end{equation}
Here $|y|=d (1,y)$ denotes the distance of $y$ from the root $1$ in
the word metric.
\end{lemma}

\begin{proof}
Denote by $H_{k} (\delta)$ the left side of \eqref{eq:snapback}; the
strategy will be to prove by induction on $k$ that for sufficiently
small $\delta >0$ the ratios $H_{k+1} (\delta)/H_{k} (\delta)$ remain
bounded as $k \rightarrow \infty$.  Consider first the sum $H_{1}
(\delta)$: by the hypothesis that $\sum_{x\in S_{m}}G_{R} (1,x)^{2}$
decays exponentially in $m$ and the symmetry $G_{r} (x,y)=G_{r} (y,x)$
of the Green's function, for all sufficiently small $\delta>0$
\begin{equation*}
	H_{1} (\delta ):=\sum_{x\in \Gamma} G_{R} (1,x)^{2} (1+\delta)^{|x|}<\infty .
\end{equation*}

Now consider the ratio $H_{k+1} (\delta)/H_{k} (\delta)$ . Fix
vertices $x_{1},x_{2},\dotsc ,x_{k}$, and for an arbitrary vertex
$y=x_{k+1}\in \Gamma $, consider its position \emph{vis a vis} the
geodesic segment $L=L (1,x_{k})$ from the root vertex $1$ to the
vertex $x_{k}$. Let $z\in L$ be
the vertex on $L$ nearest $y$ (if there is more than one, choose
arbitrarily).  By the triangle inequality,
\[
	|y|\leq |z|+d (z,y) .
\]
Because the  group $\Gamma$ is word-hyperbolic, all geodesic
triangles --- in particular, any  triangle whose sides consist of
geodesic segments from $y$ to $z$, from $z$ to $x_{k}$, and from
$x_{k}$ to $y$, or any triangle  whose sides consist of geodesic
segments from $y$ to $z$, from $z$ to $1$, and from $1$ to $y$--- are
$\Delta$-thin, for some $\Delta<\infty$ (cf.\ \cite{gromov} or
\cite{benakli-kapovich}). Hence, any geodesic segment from $x_{k}$ to
$y$ must pass within distance $8\Delta$ of the vertex $z$. Therefore,
by the Harnack and Ancona inequalities \eqref{eq:harnack} and
\eqref{eq:ancona}, for some constant
$C_{*}=C_{R}C_{\text{Harnack}}^{32\Delta}<\infty$ independent of
$y,x_{k}$,
\begin{align*}
	G_{R} (y,1)&\leq C_{*}G_{R} (y,z)G_{R} (z,1) \quad \text{and}\\
	G_{R} (y,x_{k})&\leq C_{*}G_{R} (y,z)G_{R} (z,x_{k}).
\end{align*}
On the other hand, by the log-subadditivity of the Green's function,
\[
	G_{R} (1,z)G_{R} (z,x_{k})\leq C' G_{R} (x_{k},1).
\]
It now follows that
\begin{align*}
	(1+\delta)^{|y|} G_{R} (x_{k},y)G_{R} (y,1)
	&\leq C_{*}^{2}C'(1+\delta)^{|z|+d (z,y)} G_{R} (z,x_{k}) G_{R} (z,y) G_{R} (y,z)G_{R} (z,1)\\
	&\leq C_{*}^{2}C'(1+\delta)^{|z|+d (z,y)} G_{R} (x_{k},1)G_{R}(z,y)^{2}.
\end{align*}
Denote by $\Gamma (z)$ the set of all vertices $y\in \Gamma$ such that
$z$ is a closest vertex to $y$ in the geodesic segment $L$. Then for
each $z\in L$,
\[
	\sum_{y\in \Gamma (z)} (1+\delta)^{d (z,y)}G_{R} (z,y)^{2}\leq
	\sum_{y\in \Gamma } (1+\delta)^{|y|} G_{R} (1,y)^{2}=H_{1} (\delta).
\]
Finally, because $L$ is a geodesic segment from $1$ to $x_{k}$ there
is precisely one vertex $z\in L$ at distance $n$ from $x_{k}$ for
every integer $0\leq n\leq |x_{k} |$, so $\sum_{z\in L}
(1+\delta)^{|z|} \leq C_{\delta}
(1+\delta)^{|x_{k}|}$ where $C_{\delta}= (1+\delta)/ (2+\delta)$. Therefore,
\[
	H_{k+1} (\delta) \leq C_{*}^{2}C'C_{\delta }H_{1} (\delta)H_{k} (\delta).
    \qedhere
\]
\end{proof}

\section{Critical Exponent of the Green's function at the Spectral
Radius}\label{sec:criticalExp}

In this section we prove Theorem \ref{theorem:criticalExponent}, using
the thermodynamic formalism established in the preceding sections and
the Ancona inequality \eqref{eq:ancona}.

\subsection{Reduction to a simple case}\label{ssec:strategy} Consider
first the case $x=y=1$ of Theorem~\ref{theorem:criticalExponent}.  The
system of differential equations \eqref{eq:GPrime} implies that the
growth of the derivative $dG_{r} (1,1)/dr$ as $r \rightarrow R-$ is
controlled by the growth of the quadratic sums $\sum_{x\in
\Gamma}G_{r} (1,x)^{2}$. To show that the Green's function has a
square root singularity at $r=R$, as asserted in
\eqref{eq:criticalExponent}, it will suffice to show that the
(approximate) derivative behaves as follows as $r \rightarrow R-$:

\begin{proposition}\label{proposition:eta}
For some $0<C<\infty$,
\begin{equation}\label{eq:eta}
	\eta (r):= \sum_{x\in \Gamma} G_{r} (1,x)^{2}\sim C/\sqrt{R-r}
	\quad \text{as} \;r \rightarrow R-.
\end{equation}
\end{proposition}

This will follow from Corollary~\ref{corollary:etaDE} below.  The key
to the argument is that the growth of $\eta (r)$ as $r \rightarrow
R-$ is related by Proposition~\ref{proposition:absolutelyCont} to
that of $\Pr(2\varphi_{r})$: in particular,
Proposition~\ref{proposition:pressureEqualsZero} implies that $\eta
(r) \rightarrow \infty$ as $r \rightarrow R-$, so the dominant
contribution to the sum \eqref{eq:eta} comes from vertices $x$ at
large distances from the root vertex $1$. Consequently, by equation
\eqref{eq:sphereAsymptotics},
\begin{equation*}
	\eta (r)= \sum_{m=0}^{\infty}\sum_{x\in S_{m}} G_{r}
	(1,x)^{2}\sim C (R,2) / (1-\exp \{\Pr
	(2\varphi_{r}) \})
	\quad \text{as} \;\; r \rightarrow R-.
\end{equation*}

Before beginning the analysis of $\eta (r)$ near the singularity $r=R$
we will show that the relation \eqref{eq:eta} implies similar
asymptotic behavior for the derivatives of all of the Green's
functions $G_{r} (x,y)$.

\begin{corollary}\label{corollary:greenAsymptotics}
There exist constants $0<C_{1,y}<\infty$ such that
\begin{equation}\label{eq:greenAsymptotics}
	\sum_{x\in \Gamma} G_{r} (1,x)G_{r} (x,y)\sim C_{1,y}/\sqrt{R-r}
	\quad \text{as} \;r \rightarrow R-.
\end{equation}
\end{corollary}

\begin{proof}
Recall the probability measures $\lambda_{r,m}$ on the spheres
$S_{m}$ with densities proportional to $G_{r} (1,x)^{2}$ (see
\eqref{eq:lambda-def}), and recall that these probability transfer to
probability measures, also denoted by $\lambda_{r,m}$, on
$\Sigma^{m}(E_*)$, using the correspondence $S_{m}\leftrightarrow
\Sigma^{m}(E_*)$. By Proposition \ref{proposition:absolutelyCont}, as
$m \rightarrow \infty$ the measures $\lambda_{r,m}$ converge weakly
to a probability measure $\lambda_{r}$, and this convergence is
uniform for $r\in [1,R]$, in the sense specified by
\eqref{eq:uniformWeakConvergence}. These measures are related to the
sum in \eqref{eq:greenAsymptotics} as follows:
\[
	\sum_{x\in S_{m}} G_{r} (1,x)G_{r} (x,y) = \sum_{x\in S_{m}}
	G_{r} (1,x)^{2}\frac{G_{r} (x,y)}{G_{r} (1,x)} =
	\left(\sum_{x\in S_{m}} G_{r} (1,x)^{2}\right)
	\int \frac{G_{r} (x,y)}{G_{r} (1,x)} \,d\lambda_{r,m} (x) .
\]
As $x \rightarrow \xi\in \partial \Gamma $ the ratios $G_{r}
(x,y)/G_{r} (1,x)$ converge to the Martin kernel $K_{r} (y,\xi)$, and
the convergence is uniform, by
Theorem~\ref{theorem:holderMartinKernel}. Hence, the weak convergence
$\lambda_{r,m}\Rightarrow \lambda_{r}$ implies that
\[
	\lim_{m \rightarrow  \infty} \int \frac{G_{r} (x,y)}{G_{r}
	(1,x)} \,d\lambda_{r,m} (x) :=C_{1,y} (r)
\]
exists, and the convergence is uniform for $r\in [1,R]$. Therefore,
the corollary  follows from Proposition~\ref{proposition:eta}.
\end{proof}

\subsection{Analysis of the function \texorpdfstring{$\eta (r)$}{eta(r)} near the singularity
\texorpdfstring{$r=R$}{r=R}}\label{ssec:etaAnalysis}

To analyze the behavior of $\eta (r)$ (or equivalently that of
$\Pr(2\varphi_{r})$) as $r \rightarrow R-$, we use the differential
equations \eqref{eq:GPrime} to express the derivative of $\eta (r)$
as
\begin{equation}\label{eq:derivEta}
	\frac{d\eta}{dr}=\sum_{x\in \Gamma} \left\{ \sum_{y\in \Gamma}
				    2r^{-1} G_{r} (1,x) G_{r}
				    (1,y)G_{r} (y,x)\right\}
				    -2r^{-1}G_{r} (1,x)^{2}.
\end{equation}
(Note: The implicit interchange of $d/dr$ with an infinite sum is
justified here because the Green's functions $G_{r} (u,v)$ are
defined by power series with nonnegative coefficients.) For $r
\approx R$, the sum $\sum_{x\in \Gamma }$ is  dominated by the terms
corresponding to vertices $x$ at large distances from the root $1$. Because the
second term $2r^{-1}G_{r} (1,x)^{2}$ in \eqref{eq:derivEta} remains
bounded as $r \rightarrow R-$, it is asymptotically negligible
compared to the first term $\sum_{y}$ and so we can ignore it in
proving \eqref{eq:eta}.

The strategy for dealing with the inner sum $\sum_{y\in \Gamma}$ in
\eqref{eq:derivEta} will be similar to that used in the proof of
Lemma~\ref{lemma:snapback} above. For each $x$, let $L=L (1,x)$ be the unique
geodesic segment from the root to $x$ that corresponds to a path in
the Cannon automaton, and partition the sum $\sum_{y\in \Gamma }$
according to the nearest vertex $z\in L$:
\begin{equation}\label{eq:sumPartition}
	\sum_{y\in \Gamma}=\sum_{z\in L}\sum_{y\in \Gamma (z)}
\end{equation}
where $\Gamma (z)$ is the set of all vertices $y\in \Gamma$ such that
$z$ is a closest vertex to $y$ in the geodesic segment $L$. (If for
some $y$ there are several vertices $z_{1},z_{2},\dotsc$ on $L$ all
closest to $y$, put $y\in \Gamma (z_{i})$ only for the vertex $z_{i}$
nearest to the root $1$.) By the log-subadditivity of the Green's
function and Theorem~\ref{theorem:1} (the Ancona
inequalities) there exists a constant $C<\infty$ independent of $1\leq
r\leq R$ such that for all choices of $x\in \Gamma$, $z\in L (1,x)$,
and $y\in \Gamma (z)$,
\begin{equation}
\label{eq:byAncona}
\begin{split}
	G_{r} (1,x) G_{r}(1,y)G_{r} (y,x)& \leq CG_{r} (1,z)^{2}G_{r}
	(z,x)^{2} G_{r} (z,y)^{2} \\
   &\leq CG_{r} (1,x)^{2}G_{r} (z,y)^{2};
\end{split}
\end{equation}
consequently, for each $x\in \Gamma$,
\begin{align}\label{eq:secondIneq}
		\sum_{y \in \Gamma }G_{r} (1,x) G_{r}(1,y)G_{r} (y,x)
 		&\leq \sum_{z\in L (1,x)}\sum_{y\in \Gamma (z)}
			     CG_{r} (1,x)^{2}G_{r} (z,y)^{2}\\
\notag 		&\leq \sum_{z\in L (1,x)}\sum_{y\in \Gamma}
			     CG_{r} (1,x)^{2}G_{r} (z,y)^{2}\\
\notag 		&=  CG_{r} (1,x)^{2} (|x|+1)\eta (r).
\end{align}
Proposition~\ref{proposition:ergodic2} below asserts that there is a
positive constant $\xi (r)$ such that when $C$ is replaced by $\xi
(r)$ this inequality is in fact an approximate equality for ``most''
$x\in S_{m}$, provided $m$ is large and $r$ is near $R$.  This implies
that for large $m$ the contribution to the double sum in
\eqref{eq:derivEta} with $|x|=m$ is dominated by those $x$ that are
``generic'' for the probability measure $\lambda_{r,m}$ on $S_{m}$
with density proportional to $G_{r} (1,x)^{2}$ (cf.\ %
sec.~\ref{ssec:greenSpheres}).

\begin{proposition}\label{proposition:ergodic2}
For each $r\leq R$ and each $m=1,2,\dotsc$ let $\lambda_{r,m}$ be the
probability measure on the sphere $S_{m}$ with density proportional
to $G_{r} (1,x)^{2}$. There is a continuous, positive function $\xi
(r)$ of $r\in [1,R]$ such that for each $\varepsilon >0$, and
uniformly for $1\leq r< R$,
\begin{equation}\label{eq:ergodic}
	\lim_{m \rightarrow \infty}
	\lambda_{r,m}\left\{
	x\in S_{m}\,:\, \Bigr \lvert
	     \frac{1}{m}\sum_{y\in \Gamma} G_{r} (1,y)G_{r} (y,x) /G_{r} (1,x)
	     -\xi (r)\eta (r)\Bigr \rvert >\varepsilon \eta(r)
	 \right\} =0.
\end{equation}
\end{proposition}

This will be deduced from Corollary \ref{corollary:ergodicCorollary}
--- see section~\ref{ssec:ergodic} below.  Given
Proposition~\ref{proposition:ergodic2},
Proposition~\ref{proposition:eta} and
Theorem~\ref{theorem:criticalExponent} follow easily, as we now show.

 \begin{corollary}\label{corollary:etaDE}
There exists a positive, finite constant $C$ such that as $r
\rightarrow R-$,
\begin{equation}\label{eq:etaDE}
	\frac{d\eta}{dr}\sim C \eta (r)^{3}.
\end{equation}
Consequently,
\begin{equation}\label{eq:etaAsymptotics}
	\eta (r)^{-2}\sim C (R-r)/2.
\end{equation}
\end{corollary}

\begin{proof}
We have already observed that as $r$ near $R$, the dominant
contribution to the sum \eqref{eq:derivEta} comes from vertices $x$
far from the root.  Proposition~\ref{proposition:ergodic2} and the
uniform upper bound \eqref{eq:secondIneq} on ergodic averages imply
that as $r \rightarrow R-$,
\begin{align*}
	\frac{d\eta}{dr}\sim \sum_{x\in \Gamma} \sum_{y\in \Gamma}
				    2r^{-1} G_{r} (1,x) G_{r}
				    (1,y)G_{r} (y,x)
	&\sim 2R^{-1}\xi (R)\eta (r)\sum_{m=1}^{\infty}m\sum_{x\in
	S_{m}} G_{r} (1,x)^{2}\\
	&\sim C' \eta (r) / (1- \exp \{\Pr (2\varphi_{r}) \})^{2}\\
	& \sim C \eta (r)^{3}
\end{align*}
for suitable positive constants $C,C'$. This
proves \eqref{eq:etaDE}. The relation \eqref{eq:etaAsymptotics}
follows directly from \eqref{eq:etaDE}.
\end{proof}

\subsection{Proof of Proposition~\ref{proposition:ergodic2}}\label{ssec:ergodic}

This will be accomplished by showing that the average in
\eqref{eq:ergodic} can be expressed approximately as an ergodic
average of the form \eqref{eq:erg2}, to which the result of
Corollary~\ref{corollary:ergodicCorollary} applies.  The starting
point is a version of the decomposition \eqref{eq:sumPartition}, but
using a continuous partition of unity instead of the characteristic
functions of the sets $\Gamma(z)$, since we will need a continuous
extension to the boundary of the group.

\begin{lemma}
For $K$ large enough, we can associate to any geodesic segment $L$ in
the Cayley graph of length $2K+1$ centered at $1$ a function
$\gamma_L : \Gamma \to [0,1]$ with the following properties:
\begin{enumerate}
\item The function $\gamma_L$ extends continuously to
    $\overline{\Gamma}=\Gamma\cup S^1$.
\item For any $y\in \Gamma$, if $\pi_L(y)$ is the set of points on $L$  closest
    to $y$, then $\gamma_L(y)=0$ unless $\pi_L(y)$ is contained in the
    ball of radius $K/4$ centered at $1$.
\item For any bi-infinite geodesic $L'$ the sum of the functions
    $\gamma_L$ over all subsegments $L$ of $L'$ of length $2K+1$ is
    identically equal to $1$. More
    formally, for all $y\in \Gamma$,
  \begin{equation}
  \label{eq:adds_up_to_one}
  \sum_{i\in \Z} \gamma_{L'(i)^{-1}L'[i-K, i+K]}(L'(i)^{-1}y)=1.
  \end{equation}
\end{enumerate}
\end{lemma}

If $\gamma_L$ were defined to be the function satisfying
$\gamma_L(y)=1$ if the closest point to $y$ on $L$ is $1$, and $0$
otherwise (with ties resolved as explained in Paragraph
\ref{ssec:etaAnalysis}) then the last two last properties would be
satisfied, but $\gamma_{L}$ would not in general extend continuously
to the boundary.

\begin{proof}
Let $K$ be even. For any geodesic segment $L_0$ of length $K+1$
centered around $1$, consider the set $A$ of points $y$ such that
$\pi_{L_0}(y)$ contains a point at distance at most $K/10$ of $1$,
and the set $B$ of points $y$ such that $\pi_{L_0}(y)$ contains a
point at distance at least $K/5$ of $1$. By the hyperbolicity of the
group $\Gamma$, if $K$ is large enough, the two sets $A,B$ are
disjoint, and their closures in $\overline\Gamma$ are still disjoint.
Therefore, there exists a continuous function $0\leq g_{L_0}\leq 1$
on $\overline{\Gamma}$ equal to $1$ on $\overline{A}$ and to $0$ on
$\overline{B}$.

If $L$ is a geodesic segment of length $2K+1$, let $\gamma_L$ be
equal to $g_{L[-K/2, K/2]}$ divided by the sum of the functions
$g_{\tilde L}$ along every subsegment $\tilde L$ of $L$ of length
$K+1$. Formally,
  \begin{equation*}
  \gamma_L(y) = g_{L[-K/2, K/2]}(y) / \sum g_{L(i)^{-1}L[-K/2+i, K/2+i]}(L(i)^{-1}y).
  \end{equation*}
The sum in the denominator   is $\geq 1$ on a neighborhood of the
support of $g_{L[-K/2, K/2]}$ by construction, so the function
$\gamma_L$ is well defined and continuous. All the properties of the
lemma follow easily.
\end{proof}

Let us now define for $r\in [1,R)$ a function $f_r$ on geodesic
segments $L$ through $1$, as follows. Let $a$ and $b$ be the
endpoints of $L$. If $d(1,a)\leq K$ or $d(1,b)\leq K$, let
$f_r(L)=0$. Otherwise, let
  \begin{equation*}
  f_r(L)=\eta(r)^{-1}\sum_{y\in\Gamma}\gamma_{L[-K, K]}(y)G_r(a,y)G_r(y,b)/G_r(a,b).
  \end{equation*}
By~\eqref{eq:adds_up_to_one}, for all $x\in \Gamma$,
  \begin{equation*}
  \eta(r) \sum_{z\in L(1,x)} f_r(z^{-1}L(1,x)) = \sum_{y\in\Gamma} c_x(y)G_r(1,y)G_r(y,x)/G_r(1,x),
  \end{equation*}
where the coefficient $c_x(y)\in [0,1]$ is equal to $1$ for all
points $y$ but those whose projections on $L(1,x)$ are close to $1$
or $x$. By~\eqref{eq:byAncona}, the contribution of these points is
bounded by $C \eta(r)$. Hence,
  \begin{equation*}
  \sum_{y\in\Gamma} G_r(1,y)G_r(y,x)/G_r(1,x)= \eta(r) \sum_{z\in L(1,x)} f_r(z^{-1}L(1,x)) + O(\eta(r)).
  \end{equation*}
This quantity will be estimated thanks to
Corollary~\ref{corollary:ergodicCorollary} once the following lemma
is established.

\begin{lemma}
\label{lemma:frConvergence}
The functions $f_r$ are uniformly bounded and H\"older-continuous for
$r\in [1,R)$, and they converge in the H\"older topology to a function
$f_R$ as $r\to R-$.
\end{lemma}

The notion of H\"older continuity for functions of geodesic segments was defined
in the discussion before Corollary~\ref{corollary:ergodicCorollary}.

\begin{proof}
By definition of $\gamma_L$, any geodesic segment from a point $y$
with $\gamma_L(y)>0$ to the endpoints of the geodesic segment $L$ (or
a geodesic extension of it) passes within bounded distance of $1$.
Arguing as in \eqref{eq:byAncona}, one deduces that $f_r$ is
uniformly bounded by a constant $C$.

Consider now two geodesics $L_1$ and $L_2$ around $1$, and assume
that they coincide on a neighborhood of $1$ of size $n>K$. In
particular, their restriction $L$ to the ball $B(1,K)$ coincide. Let
$a_1$ and $b_1$ (resp.\ $a_2$ and $b_2$) be the endpoints of $L_1$
(resp.\ $L_2$). For each $y\in \Gamma$ with $\gamma_L(y)>0$,
  \begin{equation*}
  \frac{G_r(a_1,y)G_r(y,b_1)/G_r(a_1,b_1)}{G_r(a_2,y)G_r(y, b_2)/G_r(a_2,b_2)}
  = \frac{ G_r(a_1,y)/G_r(a_1,b_1)}{G_r(a_2,y) / G_r(a_2,b_1)}\cdot
  \frac{G_r(y,b_1)/G_r(a_2,b_1)}{G_r(y,b_2)/G_r(a_2,b_2)}.
  \end{equation*}
Since the geodesics from $a_1$ or $a_2$ to $b_1$ or $y$ are fellow
traveling during a time at least $n-C$ by definition (see
Section~\ref{subsec:fellow_traveling}), Theorem
\ref{theorem:holderAncona} shows that the first factor is bounded by
$1+C\varrho^n$ for some $\varrho<1$. The second factor is bounded in
the same way. Multiplying by $\eta(r)^{-1}\gamma_L(y)$ and summing
over $y$, we obtain $f_r(L) \leq (1+C\varrho^n) f_r(L')$. Since $f_r$
is bounded, this yields $f_r(L)-f_r(L') \leq C\varrho^n$. Exchanging
the role of $L$ and $L'$, we get $|f_r(L)-f_r(L')|\leq C\varrho^n$.
This shows that the functions $f_r$ are uniformly H\"older-continuous.

Next we show that the functions $f_r$ converge in the H\"older
topology when $r\to R-$. It suffices to show that they converge
pointwise, because the  uniform convergence follows from
the uniform H\"older bounds, and this implies convergence in the H\"older
topology for any exponent strictly less than the initial exponent.
Fix a geodesic segment  $L$  with endpoints $a$ and $b$: we will show
that $f_r(L)$ converges when $r\to R-$. We have
  \begin{align*}
  f_r(L)
  &= \frac{1}{G_r(a,b)} \eta(r)^{-1}
  \sum_m \sum_{y\in S_m} G_r(1,y)^2 \gamma_{L[-K,K]}(y)\frac{G_r(a,y)}{G_r(1,y)} \frac{G_r(b,y)}{G_r(1,y)}
  \\&
  = \frac{1}{G_r(a,b)} \eta(r)^{-1} \sum_m \left(\sum_{y\in S_m} G_r(1,y)^2\right) \int F_r(y) \, d\lambda_{r,m}(y),
  \end{align*}
where
\[
F_r(y)=\gamma_{L[-K,K]}(y)\frac{G_r(a,y)}{G_r(1,y)}
\frac{G_r(b,y)}{G_r(1,y)}
\]
is a continuous function on $\Gamma$ which
extends continuously to $\overline{\Gamma}$. Moreover, $F_r$
converges uniformly when $r\to R-$ to a limit $F_R$.

By Proposition~\ref{proposition:absolutelyCont}, the measures
$\lambda_{r,m}$ converge as $m\to \infty$ to a measure $\lambda_r$
supported on $\partial \Gamma$, uniformly in $r\in [1,R]$. On the
other hand, the influence of bounded $m$ in the above sum tends to
$0$ when $r\to R-$ since $\eta(r)=\sum_m \left(\sum_{y\in S_m}
G_r(1,y)^2\right)$ tends to infinity. Therefore,
  \begin{equation*}
	\lim_{r \rightarrow R-}f_{r} (L)=\frac{1}{G_r(a,b)}
	\int_{\partial\Gamma} F_R \, d\lambda_R .
  \qedhere
  \end{equation*}
\end{proof}

Lemma~\ref{lemma:frConvergence} and Corollary~\ref{corollary:ergodic}
imply that  the convergence \eqref{eq:ergodic} holds with
\[
	\xi (r)=\int (f_r\circ\alpha ) \,d\mu^{\Z}_r.
\]
That $\xi (r)$ varies continuously
with $r$ for $r\leq R$ follows from the continuous dependence of the
Gibbs state $\mu^{\Z}_{r}$ with $r$ (Theorem~\ref{thm:RPF}). Thus, to
complete the proof of Proposition~\ref{proposition:ergodic2}, it
remains only to show that $\xi (R)>0$. For this it suffices to prove
that there exists $C>0$ such that for all $r$ near $R$ and for $x$
with $|x|=m$, then
\begin{equation}\label{eq:etaPositive}
	 \sum_{y\in \Gamma} G_r(1,y)G_r(y,x)/G_r(1,x) \geq Cm \eta(r).
\end{equation}
This sum can be partitioned  by
decomposing $\Gamma$ as $\bigcup \Gamma(z_j)$ for $z_j \in
L(1,x)$. Hence, to prove \eqref{eq:etaPositive} it is enough to  show
that each sufficiently long block of consecutive points
$z_j$ on the geodesic segment $L=L (1,x)$ contributes at least an
amount $C\eta(r)$ to the sum. This follows from the next lemma.

\begin{lemma}\label{lemma:density}
There exist  $K<\infty$ and $C>0$ independent of $1\leq r< R$ so that
the following is true. For any geodesic segment $L$ of length $\geq
K$  and any $K$ consecutive vertices $z_{1},z_{2},\dotsc , z_{K}$ on $L$,
\begin{equation}\label{eq:density}
	\sum_{j=1}^{K} \sum_{y\in \Gamma (z_{j})} G_{r} (z_j,y)^{2}
	\geq C \eta (r).
\end{equation}
\end{lemma}

We will deduce this from the following statement.

\begin{lemma}\label{lemma:halfplaneEst}
Assume that the Cayley graph $G^{\Gamma}$ is embedded in the
hyperbolic plane in such a way that the group identity $1\in \Gamma$
is identified with the center $O$ of the Poincar\'e disk. Then
there exists $C>0$ such that for any halfplane $H$ whose boundary is a
hyperbolic geodesic through $O$, and any $1\leq r\leq R$,
\begin{equation}\label{eq:halfplaneEst}
	\sum_{x\in \Gamma \,:\, xO\in H}  G_{r} (1,x)^{2}\geq C \eta (r).
\end{equation}
\end{lemma}

\begin{proof}
Because $\Gamma$ is co-compact, it has no parabolic elements and at
least one (and therefore infinitely many) hyperbolic
element. Furthermore, $\Gamma $ acts ergodically on
$\partial \zz{D}\times \partial \zz{D}$, so  the set
consisting of all pairs $\xi^{\pm} (x)$ of fixed points of hyperbolic
elements $x$ is dense in $\partial \zz{D}\times \partial \zz{D}$. In
particular,  for any two
antipodal points $\zeta^{+},\zeta^{-}$ on $\partial \zz{D}$ there are
hyperbolic elements $x_{n}$ such that $x_{n}^{\pm}O \rightarrow
\zeta^{\pm}$ (in the Euclidean metric on $\zz{D}\cup \partial \zz{D}$)
as $j \rightarrow \infty$.

Let $H$ be a halfplane whose boundary is a hyperbolic geodesic
$\gamma_{H}$ through $O$, and let $\zeta^{\pm}$ be antipodal points on
$\partial \zz{D}$ distinct from the endpoints of
$\gamma_{H}$. Precisely one of these antipodal points -- say
$\zeta^{+}$ -- will lie in the closure of $H$.  Let $x_{n}$ be a
sequence of hyperbolic elements such that $x_{n}^{\pm}O \rightarrow
\zeta^{\pm}$. If $y\in \Gamma$ is such that $x_{n}yO\not \in H$, for
some large $n$, then by the thin triangle property the geodesic
segment from $x_{n}O$ to $x_{n}yO$ must $\varepsilon -$shadow the
geodesic segment from $x_{n}O$ to $O$, for some constant $\varepsilon
<\infty$ independent of $x_{n}$ and $y$. Hence, by the translation
invariance of the metric, the geodesic segment from $O$ to $yO$ must
$\varepsilon -$shadow the geodesic segment from $O$ to $x_{n}^{-1}O$.
Now suppose that $\zeta^{\pm}$ and $\xi^{\pm}$ are distinct antipodal
pairs, neither coinciding with the endpoints of $\gamma_{H}$, and let
$x_{n},y_{n}$ be sequences of hyperbolic elements of $\Gamma$ such
that $x_{n}^{\pm}O \rightarrow \xi^{\pm}$ and $y_{n}^{\pm}O
\rightarrow \zeta^{\pm}$.  Then for all $n$ sufficiently large and all
$y\in \Gamma$ either $x_{n}yO\in H$ or $y_{n}yO\in H$, because $d
(x_{n}^{-1},y_{n}^{-1}) \rightarrow \infty$.

By placing three pairs of antipodal points so that the six points are
equally spaced along the circle $\partial \zz{D}$, we can arrange that
every halfplane $H$ bounded by a hyperbolic geodesic $\gamma_{H}$
through $O$  will have at least two of the six points in the interior
of the boundary arc $\partial H\subset \partial \zz{D}$. Consequently,
there exist three hyperbolic elements $h_{i}\in \Gamma$ such that for
every halfplane $H$ and every $y\in \Gamma$, at least one of the six
points $h_{i}^{\pm} yO$ lies in $H$. It now follows by the trivial
inequality $G_{r} (1,uv)\geq G_{r} (1,u)G_{r} (1,v)$ that
\begin{equation*}
	\sum_{x\in \Gamma \,:\, xO\in H}  G_{r} (1,x)^{2}\geq	
	(1\wedge \min_{i\leq 3} G_{r} (1,h_{i})^{2}) \sum_{x\in
	\Gamma}  G_{r} (1,x)^{2} .
\qedhere
\end{equation*}
\end{proof}

\begin{proof}
[Proof of Lemma~\ref{lemma:density}] Let $L$ be a geodesic segment in
the Cayley graph (relative to the graph distance) of length $\geq K$,
and let $z_{1},\dotsc ,z_{K}$ be any $K$ consecutive vertices along
$L$. By the hyperbolicity of $G^{\Gamma}$, together with the
equivalence of the word and hyperbolic metrics, there is a constant
$K<\infty$ such that for any choice of the geodesic segment $L$ and
$K$ consecutive vertices $z_{i}$ along $L$, the region
$\cup_{i=1}^{K}\Gamma(z_{i})$ will contain a hyperbolic halfplane $H$
(actually, its intersection with $\Gamma O$) whose boundary is a
hyperbolic geodesic that passes through exactly one vertex $gO$
located at a (graph) distance $\leq K$ from $z_{1}$.  The inequality
$\eqref{eq:density}$ now follows routinely from
Lemma~\ref{lemma:halfplaneEst}, by the Harnack inequality
\eqref{eq:harnack} and the translation invariance of the Green's
function.
\end{proof}

\section{Asymptotics of transition
probabilities}\label{sec:asymptotics}

Theorem \ref{theorem:localLimit}, which gives the asymptotics of
transition probabilities in a co-compact Fuchsian group, is a direct
consequence of the asymptotics of the Green's function given by
Theorem~\ref{theorem:criticalExponent} and of the following general
statement.

\begin{theorem}
\label{theorem:asymptotics}
Consider a symmetric irreducible aperiodic random walk in a countable
group $\Gamma$. Let $R$ denote the radius of convergence of the
Green's function $G_z(x,y)=\sum z^n p^n(x,y)$. Assume that there
exists $\beta >0$ such that for all $x,y\in \Gamma$,
\begin{equation}\label{eq:derivativeAsymptotics}
	 G'_r(x,y) \sim C_{x,y}/(R-r)^\beta \quad \text{as} \;\;r \uparrow R,
\end{equation}
for some $C_{x,y}>0$. Then
there exist constants $C'_{x,y}>0$ such that
\begin{equation}\label{eq:lltGeneral}
	 p^n(x,y)\sim C'_{x,y} R^{-n}n^{\beta-2} \quad \text{as} \;\;
	 n \rightarrow \infty .
\end{equation}
\end{theorem}

For the proof, we will rely on the following tauberian theorem of
Karamata (see e.g.\ \cite{regular_variation}, Corollary 1.7.3 or
\cite{feller-2}, Theorem XII.5.5).

\begin{theorem}
Let $A (z)=\sum a_n z^n$ be a power series with nonnegative
coefficients $a_{n}$ and radius of convergence $1$. Assume that, when
$s$ tends to $1$ along the real axis, $\sum a_n s^n \sim
c/(1-s)^\beta$ with $\beta>0$. Then $\sum_{k=1}^n a_k \sim c
n^{\beta}/\Gamma(1+\beta)$.
\end{theorem}

Under the assumptions of Theorem~\ref{theorem:asymptotics}, this
implies asymptotics for $\sum_{k=1}^n k R^k p^k(x,y)$. Asymptotics of
$R^n p^n(x,y)$ would readily follow if this sequence were
non-increasing (see Lemma~\ref{lemma:simpleTauberian} below). This is
not the case in general. However, we will show that for a
\emph{symmetric, aperiodic} random walk the coefficients can be
decomposed as $R^{n}p^{n} (x,y)=q_{n} (x,y)+O (e^{-\delta n})$ with
$q_{n}(x,y)$ non-increasing and $\delta >0$; this will allow us to
deduce the local limit theorem \eqref{eq:lltGeneral}.

\subsection{Spectral analysis of the transition probability
operator}
\label{ssec:spectralAnalysis} The hypothesis that the random
walk is symmetric implies that the Markov operator $\zz{P}$ of the
random walk, acting on the space $\ell^2(\Gamma)$ by $\zz{P}u(x)=\sum
p(x,y) u(y)$, is self-adjoint and bounded.

\begin{theorem}\label{theorem:spectrum}
Assume that the random walk is symmetric, irreducible, and aperiodic.
Then there exists $\varepsilon >0$ such that the spectrum of the
Markov operator $\zz{P}$ is contained in the interval $[-R^{-1}
(1+\varepsilon)^{-1} ,R^{-1}]$. Consequently, for all $x,y\in \Gamma$
the Green's function $G_{z} (x,y)$ extends holomorphically to the
doubly slit plane $\zz{C}\setminus ((-\infty ,-R (1+\varepsilon
)]\cup [R,\infty ))$.
\end{theorem}

\begin{proof}
It is well known (see, for instance, \cite{woess:book}) that the
spectrum of $\zz{P}$ is contained in the interval $[-R^{-1},R^{-1}]$,
where $R$ is the common radius of convergence of the Green's
functions. Hence, by the  spectral theorem, for any function $u\in
\ell^2(\Gamma)$ there exists a nonnegative measure $\nu=\nu_{u}$ on
$[-R^{-1}, R^{-1}]$, with total mass $\xnorm{\nu}=\xnorm{u}_{2}^{2}$,
such that for any $n\in \zz{N}$
\begin{equation}\label{eq:spectralTheorem}
  \langle u, \zz{P}^n u\rangle = \int t^n \, d\nu(t) ,
\end{equation}
and for any complex number $z$ of modulus
$|z|>R^{-1}$,
\begin{equation}\label{eq:resolvent}
	R_{u} (z):=\xclass{u, (z-\zz{P})^{-1}u}= \int \frac{1}{z-t}
	\,d\nu (t).
\end{equation}
To prove the theorem it suffices to show that for some
$\varepsilon >0$ the spectral measures $\nu_{u}$, where $u\in
\ell^2(\Gamma)$, all have support contained in $[-R^{-1}
(1+\varepsilon)^{-1} ,R^{-1}]$.

Formula \eqref{eq:resolvent} exhibits the resolvent function $R_{u}
(z)$ as the \emph{Stieltjes transform} of the measure $\nu=\nu_{u}$.
Since  $\nu$ is nonnegative and has finite total mass, its
Stieltjes transform $R_{u} (z)$ extends holomorphically to $\zz{C}\setminus
[-R^{-1},R^{-1}]$, and it satisfies the reflection
identity $R_{u} (\overline{z})=\overline{R_{u} (z)}$. According to the
\emph{Stieltjes inversion theorem} (see, e.g., \cite{wall}), for any two real
numbers $x_{1}<x_{2}$, neither of which is an atom of $\nu$,
\begin{equation}\label{eq:stieltjes}
	\nu [x_{1},x_{2}]=-\lim_{y \rightarrow 0} \Im
	\frac{1}{\pi}\int_{x_{1}}^{x_{2}}  R_{u} (x+iy) \, dx
\end{equation}
where $\Im$ denotes imaginary part. Suppose that $R_{u} (z)$ is
analytic in a neighborhood of $[x_{1},x_{2}]$. Since the function
$R_{u}$ satisfies the reflection identity, it must be real-valued on
$[x_{1},x_{2}]$, and so it follows from the inversion formula
\eqref{eq:stieltjes} that $\nu [x_{1},x_{2}]=0$. Thus, to complete the
proof of the theorem it suffices to show that there is some
$\delta  >0$ such that for every $u\in \ell^{2}
(\Gamma)$, the resolvent function $R_{u} (z)$ has an analytic
continuation to $[-R^{-1},-R^{-1}+\delta]$.

Observe that in the special case where $u=I_{\{1\}}$ is the indicator
function of the group identity, $z^{-1}R_{u} (z^{-1})=G_{z} (x,x)$ is
the Green's function, and in the case where $u=I_{\{x,y \}}$ is the
indicator of a two-point set, $z^{-1}R_{u}
(z^{-1})=2G_{z}(x,x)+2G_{z}(x,y)$. Hence, the Green's functions
extend holomorphically to $\zz{C}\setminus ((-\infty ,-R]\cup
[R,\infty ))$. According to a theorem of Cartwright
\cite{cartwright}, for any aperiodic, symmetric random walk on a
countable group the only singularity of the Green's functions $G_{z}
(x,y)$ on the circle $|z|=R$ is $z=R$. In fact, his proof shows that
there is an open neighborhood $U$ of $\{z\,:\,|z|=R \}\setminus
\{R\}$ to which all of the functions $G_{z} (x,y)$ extend
holomorphically.  It follows that for $\delta >0$ sufficiently small,
the resolvent function $R_{u}(z)$ extends holomorphically to
$[-R^{-1},-R^{-1}+\delta]$ for all functions $u$ that are indicators
of either one-point or two-point sets, and so the corresponding
spectral measures attach zero mass to this interval.  Since the
spectral measure $\nu_{u}$ of any function $u$ with finite support
can be written as a linear combination of one-point or two-point
indicators, it follows that they likewise attach zero mass to the
interval $[-R^{-1},-R^{-1}+\delta]$. Finally, for any $u\in \ell^{2}$
there is a sequence of finitely supported functions $u_{n}$ that
converge to $u$ in $\ell^{2}$, and for any such sequence the
corresponding spectral measures $\nu_{u_{n}}$ converge weakly to
$\nu_{u}$. (This follows, for instance, from
\eqref{eq:spectralTheorem}, which implies convergence of moments.)
Therefore, for any $u\in \ell^{2}$ the spectral measure has support
contained in $[-R^{-1}+\delta ,R^{-1}]$.
\end{proof}

\begin{corollary}\label{corollary:monotonicity}
Consider a symmetric, irreducible, aperiodic random walk on a countable
group, and let $R$ be the radius of convergence of the Green's
function. Then $R^n p^n(1,1) = q_n + O(e^{-\delta n})$, where $q_n$
is nonincreasing and $\delta>0$. Furthermore, for every $x\not =1$,
$R^{n}p^{n} (1,1)+R^{n}p^{n} (1,x)= q_{n} (1,x)+O(e^{-\delta n})$,
where $q_{n}(1,x)$ is nonincreasing.
\end{corollary}

\begin{proof}
Let $u$ be the indicator function of the singleton $\{1\}$, and let
$\nu =\nu_{u}$ be the corresponding spectral measure. Then by
\eqref{eq:spectralTheorem},
\[
	p^{n} (1,1) =\langle u, \zz{P}^n u\rangle = \int t^n \,
	d\nu(t) =\int t^n \, d\nu^{+}(t)+ \int t^n \, d\nu^{-}(t) ,
\]
where $\nu^{+}$ and $\nu^{-}$ are the restrictions of $\nu$ to the
positive and nonpositive reals, respectively. The sequence
$q_{n}=R^{n}\int t^{n}\,d\nu^{+} (t)$ is clearly nonincreasing in
$n$, since $\nu^{+}$ is supported by the interval $(0,R^{-1}]$.
By Theorem~\ref{theorem:spectrum}, the support of
$\nu^{-}$ is contained in  $[-R^{-1}
(1+\varepsilon)^{-1} ,0]$ for some $\varepsilon >0$. Hence,
\[
	\int t^n \, d\nu^{-}(t) =O (R^{-n} (1+\varepsilon)^{-n}).
\]
This proves the first assertion.  A similar argument proves the second
assertion, since
\[
	2p^{n} (1,1)+2p^{n} (1,x)=\int t^n \, d\nu_{1,x}(t)
\]
where $\nu_{1,x}$ is the spectral measure of the indicator function of
the two-point set $\{1,x \}$.
\end{proof}

\subsection{Proof of
Theorem~\ref{theorem:asymptotics}}\label{ssec:asymptoticsProof}

Consider first the case $x=y=1$. By
Corollary~\ref{corollary:monotonicity}, we may write $R^k p^k(1,1) =
q_k+r_k$ where $r_k$ is exponentially small and $q_k$ is
non-increasing. Thus, to prove the asymptotic formula
\eqref{eq:lltGeneral} for $x=y=1$, it suffices to prove that
\begin{equation*}
	q_{n}\sim C n^{\beta -2} \quad \text{as} \;\; n \rightarrow \infty.
\end{equation*}

For $s\in [0,1)$, let $r=Rs$ and
  \begin{equation*}
  A(s)= r G_r(1,1)' = \sum n r^n p^n(1,1) = \sum s^n \cdot n R^n  p^n(1,1).
  \end{equation*}
By hypothesis,  $A (s)\sim  C/(1-s)^\beta$ when $s\uparrow 1$. Since
$R^{k}p^{k} (1,1)=q_{k}+r_{k}$  with $r_{k}$ exponentially decaying
in $k$, it follows that $\sum_{k} k q_{k}s^{k} \sim C/(1-s)^\beta$ as
$s \uparrow 1$. Therefore, Karamata's theorem gives
\begin{equation}\label{eq:kqsumAsymptotics}
  \sum_{k=1}^n k q_k \sim C n^\beta.
\end{equation}
The desired result now follows from the next lemma.

\begin{lemma}\label{lemma:simpleTauberian}
Let $q_{n}$ be a  nonnegative sequence that satisfies
\eqref{eq:kqsumAsymptotics} for some $\beta >0$. If $q_{n}$ is
non-increasing, then as $n \rightarrow \infty$,
\[
	q_{n}\sim C\beta n^{\beta -2}.
\]
\end{lemma}

\begin{proof}
Fix $\varepsilon >0$.  Writing $S_n =
\sum_{k=0}^{n-1} kq_k$, we have
  \begin{equation*}
  \varepsilon n (1-\varepsilon)n q_n \leq \sum_{k=(1-\varepsilon)n}^{n-1} kq_n
  \leq \sum_{k=(1-\varepsilon)n}^{n-1} kq_k
  =S_n - S_{(1-\varepsilon)n} = C (n^{\beta}- (1-\varepsilon)^{\beta}n^{\beta}+o(n^{\beta})).
  \end{equation*}
Therefore,
  \begin{equation*}
  q_n \leq C n^{\beta-2} \frac{1-(1-\varepsilon)^{\beta}+o(1)}{(1-\varepsilon)\varepsilon}.
  \end{equation*}
Letting $\varepsilon$ tend to $0$, we obtain $\limsup q_n/
n^{\beta-2} \leq C\beta$. Using the interval $k\in [n,
(1+\varepsilon)n]$, we control the inferior limit in the same way,
and so we obtain $q_n \sim C\beta n^{\beta-2}$.
\end{proof}

Finally, consider the general case $x,y\in \Gamma$. Since  we have
already proved the formula \eqref{eq:lltGeneral} in the special case
$x=y$, we may assume that $x\not =y$, and by homogeneity, $x=1$. By
Corollary~\ref{corollary:monotonicity}, there is a non-increasing
sequence $q_{k} (1,y)$ and an exponentially decaying sequence $r_{k}
(1,y)$ such that $R^{k}p^{k} (1,1)+R^{k}p^{k} (1,y)=q_{k} (1,y)+r_{k}
(1,y)$. Since the formula \eqref{eq:lltGeneral} holds for $x=y=1$, to
prove it for $x=1\not =y$ it will suffice to show that
\begin{equation}\label{eq:xy}
	q_{k} (1,y)\sim (C'_{1,y} +C'_{1,1})k^{\beta -2}
\end{equation}
for some constant $C'_{1,y}$. (Note: By the Harnack inequality, the
sequences $p^{k} (1,1)$ and $p^{k} (1,y)$ are comparable, so if this
holds then $C'_{1,y}$ must be positive.)

By hypothesis \eqref{eq:derivativeAsymptotics},
\[
	\sum_{k} s^{k} kR^{k-1}(p^{k} (1,1)+p^{k} (1,y))
	\sim \frac{C_{1,1}+C_{1,y}}{(1-s)^{\beta }}
\]
as $s \uparrow 1$, and so, as in the special case considered earlier,
the generating function
$\sum_{k}kq_{k} (1,y) s^{k}$ satisfies the hypotheses of Karamata's
theorem. Thus,
\[
	\sum_{k=1}^n k q_k(1,y) \sim C n^\beta,
\]
and so the relation \eqref{eq:xy} follows from
Lemma~\ref{lemma:simpleTauberian}. \qed

\bibliographystyle{plain}
\bibliography{mainbib}

\end{document}